\documentclass[12pt]{article}

\usepackage{amssymb}
\usepackage{amsmath}
\usepackage{amsthm}
\usepackage{mathtools,xcolor,enumerate}
\usepackage[colorlinks=true,allcolors=black]{hyperref}
\usepackage[expansion=false]{microtype}

\newtheorem{theorem}{Theorem}
\newtheorem{proposition}{Proposition}[section]
\newtheorem{lemma}[proposition]{Lemma}

\theoremstyle{definition}
\newtheorem{definition}[proposition]{Definition}
\theoremstyle{remark}
\newtheorem{remark}[proposition]{Remark}

\newcommand{\R}{\mathbb{R}}
\newcommand{\C}{\mathbf{C}}
\newcommand{\cH}{\mathcal{H}}
\newcommand{\cL}{\mathcal{L}}
\newcommand{\Per}{\operatorname{Per}}
\newcommand{\diver}{\operatorname{div}}
\newcommand{\spt}{\operatorname{spt}}
\newcommand{\dist}{\operatorname{dist}}

\begin{document}

\title{Strict Stability and Strict Minimality of Regular Area-Minimizing Hypercones: A Quantitative Characterization}
\author{Gongping Niu}
\date{}
\maketitle

\begin{abstract}
Let $\C=\partial E\subset\R^{n+1}$ be a regular area-minimizing
hypercone. We first prove that $\C$ is simultaneously strictly stable and
strictly minimizing if and only if there exists $c_\C>0$ such that
\[
   \Per(F;B_R)-\Per(E;B_R)
    \geq
    c_\C
    \int_{F\mathbin\triangle E}
    \frac{\dist(x,\C)}{|x|^2}\,dx
\]
for every $R>0$ and finite-perimeter set $F$ with $F\mathbin\triangle E\Subset B_R$. Thus this intrinsic distance-weighted inequality
characterizes exactly the simultaneous strictness of the stability
and minimizing properties. Second, without either strictness
assumption, every regular area-minimizing hypercone satisfies the
scale-invariant quadratic inequality
\[
    \frac{\Per(F;B_R)-\Per(E;B_R)}{R^n}
    \geq
    c_\C
    \left(
      \frac{|F\mathbin\triangle E|}{R^{n+1}}
    \right)^2.
\]
This extends the inequality previously established for
area-minimizing Lawson cones. Finally, the area-minimizing assumption is unnecessary for our
spectral result: for every stable regular minimal hypercone, the first
Dirichlet eigenvalue $\lambda_\C^D(R)$ of the Jacobi operator on
$\C\cap B_R$ is given exactly by
\[
    \lambda_\C^D(R)
    =
    \frac{j_{b_1,1}^2}{R^2},
    \qquad
    b_1^2
    =
    \frac{(n-2)^2}{4}+\mu_1,
\]
where $\mu_1$ is the first eigenvalue of the link Jacobi operator and
$j_{b_1,1}$ is the first positive zero of the Bessel function
$J_{b_1}$. In particular, this identifies the optimal Dirichlet
spectral constant for every stable regular minimal hypercone.
\end{abstract}

\begingroup
\footnotesize
% \newpage
\tableofcontents
\endgroup
\newpage

\section{Introduction}
\label{sec:introduction}

Consider a geometric variational problem: we want to ask whether a competitor whose energy is close to the minimum is also geometrically close to an actual minimizer and, if so, at what quantitative rate. For
example, among sets of prescribed volume in $\R^{n+1}$, balls minimize the
perimeter. The sharp quantitative isoperimetric inequality states that the
normalized perimeter deficit controls from below the square of the Fraenkel
asymmetry (see \cite{fusco2008sharp,cicalese2012selection}). 

Unlike Euclidean isoperimetric minimizers, which are balls and hence smooth, area-minimizing hypersurface solutions of the Plateau problem may have singular sets of codimension at least $7$ within the hypersurface. \cite[Theorem~1.1]{inauen2018quantitative} proved a local quantitative minimality theorem for homologous integral currents sufficiently close in the flat norm to a smooth compact strictly stable critical submanifold of an elliptic
parametric functional. 

An analogous quantitative phenomenon occurs for the Plateau problem in the
smooth compact setting. Let $M$ be a smooth compact orientable minimal
hypersurface with smooth boundary, and assume that the integral current induced
by $M$ is the unique mass minimizer among integral currents with the same
boundary. \cite[Theorem~1]{de2014sharp} proved that
positivity of the first Dirichlet eigenvalue of the Jacobi operator $-(\Delta_M+|A_M|^2)$
is equivalent to a global quantitative stability inequality. More precisely,
let $M'$ be another smooth compact orientable hypersurface with the same
boundary as $M$, and let $A$ be a finite-volume Borel set whose reduced
boundary agrees with $M\mathbin\triangle M'$ up to an $\cH^n$-null set.  Their
estimate has the form
\[
  \cH^n(M')-\cH^n(M)
    \geq
    \kappa_M\min\left\{
        \cL^{n+1}(A)^2,
        \cL^{n+1}(A)^{n/(n+1)}
    \right\}.   
\]

Thus, in this setting, strict stability is equivalent to a global geometric estimate. 

The smooth compact equivalence above suggests asking whether an analogous
global characterization remains valid for singular area-minimizing
hypersurfaces.  For a regular area-minimizing hypercone, however, two
different nondegeneracy conditions arise.  Strict stability in the sense
of \cite{caffarelli1984minimal} is a spectral condition for the Jacobi
operator, whereas strict minimality in the sense of
\cite{hardt1985area} is a geometric condition (see definitions in Section~\ref{sec:preliminaries}).  

The principal question of this paper is whether these two conditions can
be captured simultaneously by a single global quantitative inequality.
We prove that this is indeed the case (Theorem~\ref{thm:weighted-characterization}).  We also show that area minimality
alone, without either strictness condition, implies an unweighted global
quadratic perimeter inequality (Theorem~\ref{thm:main}).

Let $n\geq2$, and let $E\subset\R^{n+1}$ be an open cone of locally finite
perimeter.  We assume that
\[
    \C=\partial E
\]
is a regular area-minimizing hypercone.  Thus $\C\setminus\{0\}$ is smooth,
and its link
\[
    S=\C\cap\mathbb S^n
\]
is a smooth embedded minimal hypersurface of $\mathbb S^n$.  The link $S$ is
connected and $\C$ separates
$\R^{n+1}$ into exactly two connected components (see \cite{hardt1985area}).  We orient $\C$ as the
boundary of $E$ and denote its measure-theoretic outer unit normal by
$\nu_E$.

For a set $F$ of locally finite perimeter and $R>0$, define the perimeter
deficit by
\[
    D_R(F)=\Per(F;B_R)-\Per(E;B_R).
\]
Since $E$ is locally perimeter minimizing, for  $F\mathbin\triangle E\Subset B_R$, we have $
    D_R(F)\geq 0.$

Quantitative inequalities for singular area-minimizing hypersurfaces have been established for the Lawson cones. 
\[
    M_{k,h}
    =
    \left\{(x,y)\in\R^k\times\R^h:
      \frac{|x|}{\sqrt{k-1}}
      =
      \frac{|y|}{\sqrt{h-1}}\right\},
    \qquad 2\leq k\leq h.
\]
These cones are area minimizing when $k+h\geq9$, and also when
$(k,h)=(3,5)$ or $(4,4)$.  By \cite[Theorem~5]{de2014sharp}, global
quadratic inequalities hold for every such cone except
\[
    (3,5),(2,7),(2,8),(2,9),(2,10),(2,11),
\]
up to interchanging the two factors.  By
\cite[Lemma~2 and Theorem~1]{liu2019stability}, the six remaining cases
also satisfy such inequalities. Taken together, these results establish
the quadratic inequality for every area-minimizing Lawson cone.

Those results for Lawson cones rely on their explicit geometry.  A general regular area-minimizing hypercone need not possess such symmetry. Nevertheless, the foliation theorem of Hardt and Simon (\cite{hardt1985area}) provides a canonical geometric substitute: in each component of
$\R^{n+1}\setminus\C$, the dilations of a smooth properly embedded
area-minimizing hypersurface form a foliation
\cite[Theorem~2.1]{hardt1985area}.  We use these two foliations to prove the
following estimate for every regular area-minimizing hypercone.

\begin{theorem}\label{thm:main}
There exists a constant $c_{\C}>0$ such that, for every $R>0$ and every set
$F\subset\R^{n+1}$ of locally finite perimeter satisfying
$F\mathbin\triangle E\Subset B_R$,
\begin{equation}\label{eq:main}
    \frac{D_R(F)}{R^n}
    \geq
    c_{\C}
    \left(
        \frac{\cL^{n+1}(F\mathbin\triangle E)}{R^{n+1}}
    \right)^2.
\end{equation}
\end{theorem}

\begin{remark}
We do not assume that $\C$ is either strictly stable or strictly minimizing
in Theorem~\ref{thm:main}.  The theorem therefore extends the quadratic
inequalities of \cite{de2014sharp,liu2019stability} from the Lawson family
to every regular area-minimizing hypercone.
\end{remark}

On the other hand, we want to study the first Dirichlet eigenvalue of the Jacobi operator. For $\varphi\in C^1(\C)$, set
\[
    Q_\C(\varphi)
    =\int_\C
      \left(|\nabla_\C\varphi|^2-|A_\C|^2\varphi^2\right)
      \,d\cH^n
\]
and define
\begin{equation}\label{eq:intro-ordinary-eigenvalue}
    \lambda_\C^D(R)
    =\inf_{\substack{0\ne\varphi\in C^1(\C)\\
                     \spt\varphi\Subset B_R}}
      \frac{Q_\C(\varphi)}
      {\displaystyle\int_\C\varphi^2\,d\cH^n}.
\end{equation}
As proved in
\eqref{eq:test-class-equivalence}, the same infimum is obtained from
$C_c^1((\C\setminus\{0\})\cap B_R)$.  Let
$\mu_1\leq\mu_2\leq\cdots$ be the eigenvalues of
$-(\Delta_S+|A_S|^2)$ and set
\[
    b_j=\sqrt{\frac{(n-2)^2}{4}+\mu_j}.
\]
Area minimality implies stability, so $b_1\geq0$.

Explicit positive lower bounds for the first Dirichlet eigenvalues of
area-minimizing Lawson cones were obtained in
\cite[Theorem~3]{de2014sharp} for the nonexceptional cases and in
\cite[Theorem~2]{liu2019stability} for the six remaining cases.  Our
next result extends these estimates to every regular area-minimizing
hypercone and identifies the optimal coefficient. The key observation is that expansion in the
eigenfunctions of the link Jacobi operator reduces the Rayleigh
infimum \eqref{eq:intro-ordinary-eigenvalue} to an associated one-dimensional radial problem.  The radial equation
can then be transformed into the Bessel equation of order $b_1$.  Its
first positive zero gives the candidate coefficient, and a
positive-solution identity together with endpoint cutoff
approximations shows that this coefficient is the exact Rayleigh
infimum.
\begin{theorem}
\label{thm:ordinary-Dirichlet-spectrum}
For $\nu\geq0$, let $J_\nu$ be the Bessel function of the first kind of
order $\nu$, and let $j_{\nu,1}$ denote its first positive zero.  Then, for
every $R>0$,
\begin{equation}\label{eq:intro-exact-ordinary-eigenvalue}
    \lambda_\C^D(R)
    =\frac{\lambda_\C^D(1)}{R^2}
    =\frac{j_{b_1,1}^2}{R^2}
    \geq\frac{4}{R^2}>0.
\end{equation}
Equivalently, $\frac{j_{b_1,1}^2}{R^2}$ is the optimal coefficient in
the quadratic inequality
\[
    Q_\C(\varphi)
    \geq\frac{j_{b_1,1}^2}{R^2}
      \int_\C\varphi^2\,d\cH^n,
    \qquad
    \varphi\in C^1(\C),\quad
    \spt\varphi\Subset B_R.
\]
Moreover, the sharp $|x|^{-2}$-weighted Jacobi inequality is
\[
    Q_\C(\varphi)
    \geq b_1^2
      \int_\C\frac{\varphi^2}{|x|^2}\,d\cH^n,
    \qquad
    \varphi\in C_c^1(\C\setminus\{0\}).
\]
In particular, $\C$ is strictly stable in the sense of
\cite{caffarelli1984minimal} if and only if $b_1>0$.
\end{theorem}

\begin{remark}
\hfill
\begin{itemize}
    \item The positivity of $\lambda_\C^D(R)$ does not detect strict stability (see the definition in Subsection~\ref{subsec:jacobi-bessel}): even
in the critical stable case $b_1=0$, this eigenvalue still has a
positive lower bound (for every fixed $R>0$).
\item The area-minimizing assumption is not needed for Theorem~\ref{thm:ordinary-Dirichlet-spectrum}: the conclusion remains
valid for every stable regular minimal hypercone $\C$.
\end{itemize}
    
\end{remark}  

We now turn to the weighted inequality that characterizes the two strictness conditions.

For Lawson cones, the explicit constructions in
\cite[Lemma~4.1 and equations~(4.12)--(4.13)]{de2014sharp} and
\cite{liu2019stability} exploit rotational symmetry to produce
quantitative subcalibrations whose signed divergences control
$\dist(z,M_{k,h})/|z|^2$.  Thus the distance weight is already intrinsic
to the Lawson-cone argument.  Motivated by this observation, for a
general regular area-minimizing hypercone we define
\begin{equation}\label{eq:weighted-distance-def}
    \mathcal W_R(F,E)
    =
    \int_{F\mathbin\triangle E}
    \frac{\dist(x,\C)}{|x|^2}\,dx,
    \qquad
    F\mathbin\triangle E\Subset B_R.
\end{equation} 
Moreover, if $\partial F$ is a small
normal graph of height $u$ over $\C$, the leading-order term of
$\mathcal W_R(F,E)$ is
\[
    \frac12
    \int_\C\frac{u^2}{|x|^2}\,d\cH^n,
\]
which is precisely the quantity appearing in the strict-stability
inequality.

The following theorem is the main result of the paper.  It shows
that the corresponding global weighted quantitative inequality
characterizes exactly the simultaneous strict stability and strict
minimality of the cone.

\begin{theorem}
\label{thm:weighted-characterization}
The following statements are equivalent.
\begin{enumerate}[(a)]
\item The cone $\C$ is strictly stable in the sense of
\cite{caffarelli1984minimal} and strictly minimizing in the sense of
\cite{hardt1985area}.
\item There is a constant $c_\C>0$ such that, for every $R>0$ and every
set $F$ of locally finite perimeter satisfying
$F\mathbin\triangle E\Subset B_R$,
\begin{equation}\label{eq:weighted-inequality}
    D_R(F)\geq c_\C\mathcal W_R(F,E).
\end{equation}
\end{enumerate}
\end{theorem}

The forward implication follows the quantitative-subcalibration
strategy of \cite{de2014sharp,liu2019stability}.  The new point is to
construct, without relying on Lawson symmetry, a unit subcalibration
whose signed divergence is comparable to the intrinsic weight
$\dist(x,\C)/|x|^2$.

For the converse implication, we use the compactly supported normal variations
to recover a positive $|x|^{-2}$-weighted Jacobi gap and hence strict
stability.  Once strict stability is known, the root criterion in
\cite[Theorem~3.2]{hardt1985area} shows that failure of strict minimality
produces a foliation leaf with the fast indicial decay.  Cutting off such
a leaf and joining it to the cone gives competitors whose perimeter
deficits tend to zero while their weighted distances remain bounded away
from zero, contradicting the weighted quantitative inequality.

Next we will show that the weighted
quantitative inequality (in Theorem \ref{thm:weighted-characterization}) implies the unweighted quadratic inequality (in Theorem \ref{thm:main}),
while the converse fails in general. 

\subsection*{From the weighted inequality to the unweighted quadratic inequality}
\label{subsec:weighted-to-unweighted}

Assume that condition\/ {\rm (b)} of
Theorem~\ref{thm:weighted-characterization} holds. For $x\ne0$, write
\[
    q_\C(x):=\frac{\dist(x,\C)}{|x|^2}.
\]
% We first establish a distributional estimate for this weight.  Let
% $S=\C\cap\mathbb S^n$.  Since $S$ is a smooth compact embedded
% hypersurface of $\mathbb S^n$, the tubular-neighborhood estimate gives
% a constant $C_0=C_0(\C)$ such that
% \[
%     \cH^n\left(
%         \left\{\theta\in\mathbb S^n:
%         \dist(\theta,\C)\leq s\right\}
%     \right)
%     \leq C_0\min\{s,1\}
% \]
% for every $s>0$.  Here the distance to $\C$ is comparable, near $S$, to
% the spherical distance to $S$.

By the homogeneity of $\C$, for any $\theta\in\mathbb S^n$,
\[
    \dist(r\theta,\C)=r\dist(\theta,\C),
    \qquad
    q_\C(r\theta)=\frac{\dist(\theta,\C)}{r}.
\]
Consequently, polar coordinates yield that for every $t>0$,
\begin{align}
    \cL^{n+1}\bigl(\{x\in B_R:q_\C(x)\leq t\}\bigr)
    &=
    \int_0^R r^n
    \cH^n\left(
        \left\{\theta\in\mathbb S^n:
        \dist(\theta,\C)\leq tr\right\}
    \right)\,dr \notag\\
    &\leq
    C_1\int_0^R r^n\min\{tr,1\}\,dr \notag\\
    &\leq C_1 tR^{n+2},
    \label{eq:weight-sublevel-estimate}
\end{align}
for some $C_1=C_1(\C)>0$.

Now set
\[
    A=F\mathbin\triangle E,
    \qquad
    m=\cL^{n+1}(A).
\]
If $m=0$, there is nothing to prove.  Otherwise, choose
\[
    t_0=\frac{m}{2C_1R^{n+2}}.
\]
By \eqref{eq:weight-sublevel-estimate},
\[
    \cL^{n+1}\bigl(A\cap\{q_\C>t_0\}\bigr)
    \geq
    m-\cL^{n+1}\bigl(\{q_\C\leq t_0\}\cap B_R\bigr)
    \geq\frac{m}{2}.
\]
It follows that
\begin{align}
    \mathcal W_R(F,E)
    =\int_A q_\C(x)\,dx
    &\geq
    t_0\,
    \cL^{n+1}\bigl(A\cap\{q_\C>t_0\}\bigr) \notag\\
    &\geq
    \frac{m^2}{4C_1R^{n+2}}.
    \label{eq:weighted-controls-volume}
\end{align}
Combining this with
\eqref{eq:weighted-inequality}, we obtain
\[
    D_R(F)
    \geq
    c_\C\mathcal W_R(F,E)
    \geq
    c_\C'
    \frac{\cL^{n+1}(F\mathbin\triangle E)^2}{R^{n+2}}.
\]
Thus the weighted quantitative inequality implies the unweighted
quadratic inequality.

\begin{remark}
\label{rem:plane-comparison}
The converse implication is not true in general. However,
Theorem~\ref{thm:main} is stronger in scope: it proves the unweighted
quadratic inequality for every regular area-minimizing hypercone,
without assuming either strict stability or strict minimality.
Let $\C=\R^2\times\{0\}$ and $E=\{x_3<0\}$.  Then
$\mu_1=0$, so the plane is stable but not strictly stable.  By
\cite[p.~116]{hardt1985area}, it is also not strictly minimizing.
Theorem~\ref{thm:main} nevertheless applies, and
Theorem~\ref{thm:ordinary-Dirichlet-spectrum} gives
$\lambda_\C^D(R)\geq4/R^2>0$.

The failure of the weighted inequality is also visible directly.  Choose
$\eta\in C_c^\infty([0,\infty);[0,1])$ with $\eta=1$ on $[0,1]$ and
$\eta=0$ on $[2,\infty)$, denote the stretching function $u_R(y)=\eta(|y|/R)$, and let
$$
    F_R=\{(y,s)\in\R^2\times\R:s<u_R(y)\}.
$$
Then $F_R\mathbin\triangle E\Subset B_{3R}$ for large $R$, and
$$
    D_{3R}(F_R)
    =\int_{\R^2}\left(\sqrt{1+|\nabla u_R|^2}-1\right)dy
    \leq C.
$$
On the other hand,
\begin{align*}
    \mathcal W_{3R}(F_R,E)
    &=\int_{\R^2}\int_0^{u_R(y)}
      \frac{s}{|y|^2+s^2}\,ds\,dy\\
    &\geq c\int_1^R\frac{dr}{r}
      \geq c\log R.
\end{align*}
Thus no uniform positive constant can hold.
\end{remark}

% We briefly describe the proofs.  The unit normals to the foliation leaves
% form a divergence-free calibration $X$.  On each side we construct
% a field $Y_\pm$, tangent to the leaves, with $\diver Y_\pm=1$ and
% $|Y_\pm(x)|\leq C_\C|x|$.  The resulting area and volume flux identities
% prove Theorem~\ref{thm:main}.  Theorem~\ref{thm:ordinary-Dirichlet-spectrum}
% follows from separation of variables, the ground-state transformation
% $v=r^{(n-2)/2}\varphi$, and a weighted one-dimensional estimate.  For
% Theorem~\ref{thm:weighted-characterization}, strict stability and strict minimality
% allow us to solve a weighted tangential divergence equation on each
% foliation leaf.  Its zero-homogeneous extension gives a quantitative
% subcalibration and hence the weighted inequality.  Conversely, normal
% variations give strict stability, while cutting off a hypothetical
% fast-decaying foliation leaf contradicts the weighted inequality and
% therefore gives strict minimality.

\paragraph{Acknowledgments.} I would like to thank Guido De Philippis, Martin Li, Zhenhua Liu, Francesco Maggi, Luca Spolaor, and Zhihan Wang for their encouragement and helpful comments on the manuscript. I am particularly grateful to Zhihan Wang for helpful discussions concerning Theorem~\ref{thm:main} at the beginning of this project.
\section*{Disclosure of AI tools}

This work made substantial use of OpenAI's GPT-5.6 Ultra as an interactive aid in developing research ideas, drafting and revising the manuscript, and checking proofs. I formulated the problem, directed the investigation, identified the relevant literature and possible strategies, and selected the approaches to pursue. 

Before using the model, I had developed an approach to proving Theorem~\ref{thm:main} and establishing the sufficiency direction of Theorem~\ref{thm:weighted-characterization}, but this approach required the regular cone to be both strictly stable and strictly minimizing. Through iterative discussions with the model, I identified and developed an alternative method that establishes Theorem~\ref{thm:main} for arbitrary regular area-minimizing cones and proves the necessity direction of Theorem~\ref{thm:weighted-characterization}. I independently verified all mathematical statements and proofs and take full responsibility for the final content of the paper.

\section{Notation and Preliminaries}
\label{sec:preliminaries}

We use the same notation as in Section~\ref{sec:introduction} for the cone $\C$.  We first recall the foliation and asymptotic results from \cite[Theorems~2.1 and~3.2]{hardt1985area} used below, and then cover the Jacobi spectral decomposition, the radial variational reduction, and the foliation-coordinate identities needed in the proofs.

\subsection*{Hardt--Simon foliations}

The two components of $\R^{n+1}\setminus\C$ are
$$
    U_+=\R^{n+1}\setminus\overline E,
    \qquad
    U_-=E.
$$
By \cite[Theorem~2.1]{hardt1985area}, there are smooth properly embedded area-minimizing
hypersurfaces $\Sigma_\pm\subset U_\pm$, normalized by
$$
    \dist(0,\Sigma_\pm)=1,
$$
such that
\begin{equation}\label{eq:foliation-map}
    \Phi_\pm:(0,\infty)\times\Sigma_\pm\longrightarrow U_\pm,
    \qquad
    \Phi_\pm(t,z)=tz,
\end{equation}
is a diffeomorphism.  Thus $\{t\Sigma_\pm:t>0\}$ foliates $U_\pm$. Each leaf is connected and intersects
every ray uniquely and transversely.
We orient $\Sigma_\pm$ by choosing the unit normals $\nu_\pm$ so that
$$
    \pm x\cdot \nu_\pm(x)>0
    \qquad\text{for every }x\in\Sigma_\pm.
$$

We define
\begin{equation}\label{eq:signed-foliation}
    S_\lambda
    =\begin{cases}
        \lambda\Sigma_+,&\lambda>0,\\
        \C,&\lambda=0,\\
        |\lambda|\Sigma_-,&\lambda<0.
      \end{cases}
\end{equation}
Thus $\Sigma_\pm=S_{\pm1}$ and
$S_{\pm t}=t\Sigma_\pm$ for $t>0$.  These sets are pairwise disjoint and
cover $\R^{n+1}$; for each sign,
$\{S_{\pm t}:t>0\}$ is the smooth foliation of $U_\pm$, and
$S_0=\C$.  For $t>0$ and $z\in\Sigma_\pm$, the normal on the
corresponding leaf is
$$
    \nu_{\pm t}(tz)=\nu_\pm(z),
$$
and $\nu_0=\nu_E$ on $S_0\setminus\{0\}$.

\subsection*{Jacobi spectrum and the radial eigenvalue problem}
\label{subsec:jacobi-bessel}

 Let
$\{w_j\}_{j=1}^\infty$ be an orthonormal eigenbasis of
$-(\Delta_S+|A_S|^2)$, with corresponding eigenvalues
\[
    \mu_1\leq\mu_2\leq\cdots,
\]
and set
\[
    b_j
    =
    \sqrt{\frac{(n-2)^2}{4}+\mu_j}.
\]
For $\varphi\in C_c^1(\C\setminus\{0\})$, set
\begin{align}
    Q_\C(\varphi)
    &=\int_\C
      \left(|\nabla_\C\varphi|^2-|A_\C|^2\varphi^2\right)
      \,d\cH^n.\label{eq:cone-Jacobi-form}
\end{align}

By \cite[Theorem 4.5]{caffarelli1984minimal} and \cite[Lemma 2.1]{wang2020deformations},
\begin{equation}\label{eq:sharp-cone-Hardy-bound}
    Q_\C(\varphi)
    \geq b_1^2\int_\C\frac{\varphi^2}{|x|^2}\,d\cH^n,
\end{equation}
and the constant $b_1^2$ is optimal.  Consequently, stability is equivalent
to 
\begin{equation}\label{eq:stability-mu}
    \mu_1\geq-\frac{(n-2)^2}{4}.
\end{equation}
For convenience, we denote
$$
    \gamma_\pm
    =-\frac{n-2}{2}\pm b_1.
$$
Note that $\mu_1
    \leq 0$ and
\begin{equation}\label{eq:gamma-lower}
    \gamma_-
    =-\frac{n-2}{2}-\sqrt{\frac{(n-2)^2}{4}+\mu_1}
    \geq-(n-2).
\end{equation}

We call $\C$ \emph{strictly stable} if the inequality in
\eqref{eq:stability-mu} is strict. 
If equality in \eqref{eq:stability-mu} holds, we say that
$\C$ is stable but not strictly stable; in that case the two indicial
roots coincide and
$$
    \gamma_- =\gamma_+=-\frac{n-2}{2}.
$$

We next record the polar decomposition of the Jacobi form and reduce the
ordinary Dirichlet problem to a one-dimensional radial variational
problem.

Fix
$\varphi\in C_c^1(\C\setminus\{0\})$, choose $R>0$ such that
$\spt\varphi\Subset B_R$, and write
\[
    \varphi(r,\omega)
    =
    \sum_{j=1}^\infty u_j(r)w_j(\omega).
\]
Since
\[
    d\cH^n_\C
    =
    r^{n-1}\,dr\,d\cH^{n-1}_S,
\]
we have
\begin{equation}\label{eq:polar-quadratic-form}
    Q_\C(\varphi)
    =
    \sum_j\int_0^R
      \left(
        r^{n-1}|u_j'|^2+\mu_jr^{n-3}u_j^2
      \right)\,dr
\end{equation}
and
\begin{equation}\label{eq:Dirichlet-mode-L2}
    \int_{\C\cap B_R}\varphi^2\,d\cH^n
    =
    \sum_j\int_0^Rr^{n-1}u_j^2\,dr.
\end{equation}

Set the one-dimensional variational problem
\begin{equation}\label{eq:first-radial-Rayleigh-quotient}
    \Lambda_1(R)
    :=
    \inf_{0\ne u\in C_c^1((0,R))}
    \frac{\displaystyle
      \int_0^R
      \left(r^{n-1}|u'|^2+\mu_1r^{n-3}u^2\right)\,dr}
    {\displaystyle\int_0^Rr^{n-1}u^2\,dr}.
\end{equation}
Since $\mu_j\geq\mu_1$, for every $u\in C_c^1((0,R))$,
\[
    \int_0^R
    \left(r^{n-1}|u'|^2+\mu_jr^{n-3}u^2\right)\,dr
    \geq
    \int_0^R
    \left(r^{n-1}|u'|^2+\mu_1r^{n-3}u^2\right)\,dr.
\]
Considering test functions of the form
$u(r)w_1(\omega)$, we obtain
\begin{equation}\label{eq:punctured-first-mode-reduction}
    \inf_{\substack{0\ne\varphi\in
      C_c^1((\C\setminus\{0\})\cap B_R)}}
    \frac{Q_\C(\varphi)}
    {\displaystyle\int_\C\varphi^2\,d\cH^n}
    =
    \Lambda_1(R).
\end{equation}

Denote $ \alpha=\frac{n-2}{2}.$
For $u\in C_c^1((0,R))$, integration by parts gives
\begin{equation}\label{eq:original-radial-Hardy-identity}
    \int_0^R
    \left(r^{n-1}|u'|^2+\mu_1r^{n-3}u^2\right)\,dr
    =
    \int_0^Rr\left|\left(r^\alpha u\right)'\right|^2\,dr
    +b_1^2\int_0^Rr^{n-3}u^2\,dr.
\end{equation}
The elementary estimate
\[
    |r^\alpha u(r)|^2
    \leq \int_r^R \frac1s \,ds
    \int_r^R
      s\left|\left(s^\alpha u(s)\right)'\right|^2\,ds\leq 
    \log\frac Rr
    \int_r^R
      s\left|\left(s^\alpha u(s)\right)'\right|^2\,ds
\]
and Fubini's theorem yield
\begin{align*}
    \int_0^R r^{n-1}u(r)^2\,dr
&\leq
\int_0^R
r\log\frac Rr
\int_r^R
s\left|
\left(s^\alpha u(s)\right)'
\right|^2ds\,dr\\
&=
\int_0^R
s\left|
\left(s^\alpha u(s)\right)'
\right|^2
\left(
\int_0^s r\log\frac Rr\,dr
\right)ds.
\end{align*}
Since 
$$\int_0^s r\log\frac Rr\,dr
=
\frac{s^2}{2}\log\frac Rs+\frac{s^2}{4}
\leq
\frac{R^2}{4}. $$
We get
\[
\int_0^Rr^{n-1}u^2\,dr
\leq
\frac{R^2}{4}
\int_0^Rr
\left|
\left(r^\alpha u\right)'
\right|^2dr.
\]
Since $b_1^2\geq0$, it follows from
\eqref{eq:original-radial-Hardy-identity} that
\begin{equation}\label{eq:radial-positive-lower-bound}
    \Lambda_1(R)\geq\frac4{R^2}>0.
\end{equation}
This already provides a lower bound in Theorem~\ref{thm:ordinary-Dirichlet-spectrum}; we will establish the sharp lower bound in Section~\ref{sec:proof-idea}.

\subsection*{Strict minimality and asymptotic estimates}

By \cite[Theorem~2.1 and (1.9)]{hardt1985area}, each
$\Sigma_\pm$ is a normal graph over $\C$ outside a sufficiently large ball, i.e.,
there is an $R_0>0$ such that for every $(r,\omega)\in(R_0,\infty)\times S$,
\begin{equation}\label{eq:end-parametrization}
    G_\pm(r,\omega)
    =r\omega+h_\pm(r,\omega)\nu_E(\omega).
\end{equation}

For each sign, let
$J_\pm:(R_0,\infty)\times S\longrightarrow(0,\infty)$ be the area
density of the parametrization $G_\pm$, i.e.,
$$
    G_\pm^*(d\cH^n_{\Sigma_\pm})
    =J_\pm(r,\omega)\,dr\,d\cH^{n-1}_S(\omega).
$$

For $R>R_0$, we write
\begin{equation}\label{eq:partial_r}
    \partial_r=G_{\pm*}(\partial/\partial r)=\partial_rG_\pm.
\end{equation}
If $u\in C^1((R_0,\infty)\times S)$, then 
\begin{equation}\label{eq:end-radial-divergence}
    \diver_{\Sigma_\pm}(u\partial_r)(G_\pm(r,\omega))
    =\frac{1}{J_\pm(r,\omega)}
      \partial_r\bigl(J_\pm u\bigr)(r,\omega).
\end{equation}

Next we introduce the notion of strict minimality from \cite[Definition~3.1]{hardt1985area}.  Let
$\C_1$ be the multiplicity-one current carried by $\C\cap B_1$, with the
orientation induced by $\nu_E$.

\begin{definition}
\label{def:strict-minimizing}
The cone $\C$ is \emph{strictly minimizing} if there is $\Theta_\C>0$
such that
\begin{equation}\label{eq:HS-strict-minimizing-definition}
    \mathbf M(\C_1)
    \leq \mathbf M(T)-\Theta_\C\varepsilon^n
\end{equation}
whenever $\varepsilon>0$ and $T$ is an integer-multiplicity
$n$-current satisfying
$$
    \spt T\subset\R^{n+1}\setminus B_\varepsilon,
    \qquad
    \partial T=\partial\C_1.
$$
\end{definition}
If $\C$ is strictly stable, then $\gamma_-<\gamma_+$. For each side, write $h=h_\pm$; there are $a\ne0$ and $\beta=\gamma_-$ or $\gamma_+$ such that
\begin{equation}\label{eq:power-asymptotic}
    h(r,\omega)
    =a r^{\beta}w_1(\omega)
      +O_2(r^{\beta-\varepsilon}).
\end{equation}
If $\C$ is stable but not strictly stable, then, for each side,
\begin{equation}\label{eq:log-asymptotic}
    h(r,\omega)
    =r^{-(n-2)/2}(a+b\log r)w_1(\omega)
      +O_2(r^{-(n-2)/2-\varepsilon}),
\end{equation}
where $(a,b)\ne(0,0)$.  The notation
$O_2(r^\alpha)$ means that the remainder and its first two derivatives on $\C$ are bounded, uniformly in $\omega$, by
$Cr^{\alpha-j}$ for $j=0,1,2$ (see \cite[Lemma~4.6(2)]{wang2020deformations}). Hence, in either case,
\begin{equation}\label{eq:J-two-sided}
   J_\pm(r,\omega)= r^{n-1}(1+o(1)).
\end{equation}

With these orientations, set
\begin{equation}\label{eq:psi-def}
    \psi_\pm(z):=\pm z\cdot\nu_\pm(z)>0
    \qquad\text{on }\Sigma_\pm.
\end{equation}
By \cite[Theorem~2.1 and Remark~2.2(2)--(3)]{hardt1985area}, the normal
component $z\cdot\nu_\pm(z)$ of the infinitesimal dilation is a Jacobi field
with a strict sign and is asymptotically comparable to the height of the
normal graph: for all sufficiently
large $r$,
\begin{equation}\label{eq:HS-psi-comparison}
    c_\pm|h_\pm(r,\omega)|
    \leq\psi_\pm(G_\pm(r,\omega))
    \leq C_\pm|h_\pm(r,\omega)|,
\end{equation}
for some $0<c_\pm<C_\pm<\infty$, uniformly in $\omega$. 

In the case of $\C$ being strictly stable, we have the following
consequences of the Hardt--Simon and Wang asymptotics.

\begin{proposition}[{\cite[Theorem 3.2]{hardt1985area}}]
\label{prop:HS-root-criterion}
Assume that $\C$ is strictly stable, and choose $w_1>0$.  Then $\C$ is strictly minimizing if
and only if both foliation leaves have the slow exponent, namely
\begin{equation}\label{eq:slow-root-both-sides}
    \beta=\gamma_+
    =-\frac{n-2}{2}+b_1.
\end{equation}
If these equivalent conditions hold, put $\gamma=\gamma_+$.  There are
$\varepsilon>0$ and $R_1>R_0$ such that, for each
$\sigma\in\{+,-\}$, there is $a_\sigma\ne0$ satisfying
\begin{equation}\label{eq:slow-weighted-C2}
    h_\sigma
    =a_\sigma r^\gamma w_1+O_2(r^{\gamma-\varepsilon}).
\end{equation}
\end{proposition}

\subsection*{Foliation coordinates and homogeneous fields}

Next, we introduce the scaling identities used in the construction of the
ambient vector fields. Every $S_\lambda$ with $\lambda\ne0$ is a smooth minimal hypersurface.  For $t>0$, $z\in\Sigma_\pm$,
$\tau\in\R$, and
$V\in T_z\Sigma_\pm$, differentiation of the map \eqref{eq:foliation-map} gives
\begin{equation}\label{eq:Phi-differential}
    (D\Phi_\pm)_{(t,z)}(\tau,V)=\tau z+tV.
\end{equation}
Hence,
\begin{equation}\label{eq:scaled-leaf-geometry}
    T_{tz}S_{\pm t}=T_{tz}(t\Sigma_\pm)=tT_z\Sigma_\pm,
    \qquad
    \nu_{\pm t}(tz)=\nu_{t\Sigma_\pm}(tz)=\nu_\pm(z).
\end{equation}

Let $e_1,\ldots,e_n$ be an oriented orthonormal basis of
$T_z\Sigma_\pm$.  
The coordinate vectors under $D\Phi_\pm$ are
$z,te_1,\ldots,te_n$.  Thus its Jacobian is
\begin{equation}\label{eq:Phi-Jacobian}
\begin{aligned}
    J_{\Phi_\pm}(t,z)
    &=\left|\det(z,te_1,\ldots,te_n)\right|\\
    &=t^n|z\cdot\nu_\pm(z)|
      =t^n\psi_\pm(z).
\end{aligned}
\end{equation}

Next we compute the divergence formulas for $k$-homogeneous vector fields used below.
Let $Z\in C^1(\Sigma_\pm;T\Sigma_\pm)$ and, for $k\in\R$, define the $k$-homogeneous vector fields 
\[
    Y^{(k)}(tz)=t^kZ(z).
\]
In the coordinates of $\Phi_\pm$, by \eqref{eq:Phi-differential}, we have
\begin{align*}
      (D\Phi_\pm)_{(t,z)}
    \bigl(0,t^{k-1}Z(z)\bigr)
    &=
    t\bigl(t^{k-1}Z(z)\bigr) \\
    &=
    t^kZ(z) \\
    &=
    Y^{(k)}(tz).
\end{align*}

Therefore, we have
$$
\begin{aligned}
    (\diver Y^{(k)})(tz)
    &=
    \frac{1}{t^n\psi_\pm(z)}
    \diver_{\Sigma_\pm}
    \left(
        t^n\psi_\pm(z)t^{k-1}Z(z)
    \right) \\
    &=\frac{t^{k-1}}{\psi_\pm(z)}
      \diver_{\Sigma_\pm}(\psi_\pm Z)(z).
\end{aligned}
$$

In particular, the degree-one case $Y(tz)=tZ(z)$ gives
\begin{equation}\label{eq:scaled-field-divergence}
    (\diver Y)(tz)
    =\frac{1}{\psi_\pm(z)}
      \diver_{\Sigma_\pm}(\psi_\pm Z)(z).
\end{equation}
Define the foliation parameter
\begin{equation}\label{eq:lambda-def}
    \lambda_\pm
    \colon U_\pm\longrightarrow(0,\infty),
    \qquad
    \lambda_\pm(tz)=t.
\end{equation}
Equation \eqref{eq:Phi-differential} implies
$$
    (d\lambda_\pm)_{tz}(\tau z+tV)=\tau.
$$
In particular, $\lambda_\pm$ has no critical points,
$\{\lambda_\pm=s\}=s\Sigma_\pm=S_{\pm s}$, and
\begin{equation}\label{eq:lambda-gradient}
    \nabla\lambda_\pm(tz)
    =\frac{\nu_\pm(z)}{z\cdot\nu_\pm(z)}
    =\pm\frac{\nu_\pm(z)}{\psi_\pm(z)}.
\end{equation}

The unit normals of the leaves converge to the unit normal of the cone, i.e., if
$x_j=t_jz_j\in U_\pm$ converges to
$x\in\C\setminus\{0\}$, where $z_j\in\Sigma_\pm$, then, because $\Sigma_\pm$ is $C^2$-asymptotic to $\C$, we have
\begin{equation}\label{eq:leaf-normal-trace}
    t_j\longrightarrow0,
    \qquad
    |z_j|\longrightarrow\infty,
    \qquad
    \nu_\pm(z_j)\longrightarrow\nu_E(x).
\end{equation}

\section{Proofs of the main theorems}
\label{sec:proof-idea}

In this section, we prove the three main results while using several geometric identities established later.  The proof of Theorem~\ref{thm:main} uses the fields constructed in Section~\ref{sec:field-construction}.  For the sufficiency direction of Theorem~\ref{thm:weighted-characterization}, we use Lemma~\ref{lem:weighted-subcalibration}, whose construction is carried out in Section~\ref{sec:weighted-characterization}; the necessity direction is proved completely here.

\subsection{Proof of Theorem \ref{thm:main}}

We first consider the unit normal field of the foliation leaves.
Every $x\in\R^{n+1}\setminus\{0\}$ lies on a unique member of the signed
family $\{S_\lambda\}_{\lambda\in\R}$.  We define $X(x)$ to be the
oriented unit normal $\nu_\lambda(x)$ to that leaf.  Thus $X$ records the
normal direction of the entire foliation in a single field.  In Lemma \ref{lem:calibration}, we see that 
\begin{align*}
    X|_{S_\lambda}&=\nu_\lambda
        &&\text{for }\lambda\ne0,\\
    X&=\nu_E
        &&\text{on }\C\setminus\{0\},\\
    \diver X&=0
        &&\text{in }\mathcal D'(\R^{n+1}).
\end{align*}
Thus $X$ is the calibration determined by these foliations.

Since $\diver X=0$,
Lemma~\ref{lem:calibration-identity} (see also \cite[Proposition~4.1]{de2014sharp}) gives, for every set $F$ of locally
finite perimeter satisfying $F\mathbin\triangle E\Subset B_1$,
\begin{align*}
    D_1(F)
    &=\int_{\partial^*F\cap B_1}
      (1-X\cdot\nu_F)\,d\cH^n\\
    &=\frac12\int_{\partial^*F\cap B_1}
      |\nu_F-X|^2\,d\cH^n.
\end{align*}

Next, we use a second vector field $Y_\pm$ on each side of the cone (Lemma \ref{lem:ambient-Y}).  For each
sign $\pm$, we construct a smooth vector field $Y_\pm$ on $U_\pm$ such
that
\begin{align*}
    \diver Y_\pm&=1&&\text{in }U_\pm,\\
    Y_\pm\cdot X&=0&&\text{in }U_\pm,\\
    |Y_\pm(x)|&\leq C_{\C}|x|&&\text{for }x\in U_\pm.
\end{align*}
Since $X$ is the unit normal to the leaves, the identity
$Y_\pm\cdot X=0$ means that $Y_\pm$ is tangent to every leaf $S_\lambda$
on the corresponding side. And $\diver Y_\pm=1$ recovers volume
as a boundary flux.
Therefore, let $U$ be either $U_+$ or $U_-$, let $Y$ be the corresponding field, and suppose that $A\subset U$ has locally finite perimeter and $\spt\chi_A\Subset B_1$.  Then by Lemma \ref{lem:volume-flux},
\begin{equation*}
    \cL^{n+1}(A)
    =\int_{\partial^*A\cap U}Y\cdot\nu_A\,d\cH^n.
\end{equation*}
Thus the field $Y_\pm$ measures the volume on each side of the cone.

We now prove Theorem~\ref{thm:main} using the properties above.  We
consider the case $R=1$.  The general case then follows by scaling.

\begin{proof}[Proof of Theorem~\ref{thm:main}]
Write
$$
    D=D_1(F),
    \qquad
    A_+=F\setminus E,
    \qquad
    A_-=E\setminus F.
$$
Both $A_+$ and $A_-$ have locally finite perimeter.
Since the cone $\C$ has zero $(n+1)$-dimensional measure, up to null sets
we may write
$$
    A_+=F\cap U_+,
    \qquad
    A_-=F^c\cap U_-.
$$
Moreover,
$\spt\chi_{A_\pm}\Subset B_1$.  By the locality of the reduced boundary \cite[Remark 15.2]{maggi2012sets},
up to $\cH^n$-null sets,
$$
\begin{aligned}
    \partial^*A_+\cap U_+&=\partial^*F\cap U_+,
    &\nu_{A_+}&=\nu_F,\\
    \partial^*A_-\cap U_-&=\partial^*F\cap U_-,
    &\nu_{A_-}&=-\nu_F.
\end{aligned}
$$
Applying Lemma \ref{lem:volume-flux} separately to $A_+$ and $A_-$, we get 
\begin{align}
    \cL^{n+1}(A_+)
    &=\int_{\partial^*F\cap U_+}
      Y_+\cdot\nu_F\,d\cH^n,\label{eq:Aplus}\\
    \cL^{n+1}(A_-)
    &=-\int_{\partial^*F\cap U_-}
      Y_-\cdot\nu_F\,d\cH^n.\label{eq:Aminus}
\end{align}
Hence, adding
\eqref{eq:Aplus} and \eqref{eq:Aminus},
\begin{align*}
    \cL^{n+1}(F\mathbin\triangle E)
    &=\cL^{n+1}(A_+)+\cL^{n+1}(A_-)\\
    &=\int_{\partial^*F\cap U_+}Y_+\cdot\nu_F\,d\cH^n
      -\int_{\partial^*F\cap U_-}Y_-\cdot\nu_F\,d\cH^n\\
    &=\int_{\partial^*F\cap U_+}
        Y_+\cdot(\nu_F-X)\,d\cH^n
      -\int_{\partial^*F\cap U_-}
        Y_-\cdot(\nu_F-X)\,d\cH^n.
\end{align*}
The last equality follows from $Y_\pm\cdot X=0$.

Let $K=\spt(\chi_F-\chi_E)\Subset B_1$. Consider $O=(U_+\cup U_-)\setminus K.$
The sets $F$ and $E$ agree almost
everywhere in $O$. Hence,
$$
    \operatorname{Per}(F;O)
    =
    \operatorname{Per}(E;O)
    =
    0,
$$
where the last equality follows because $E$ is one of
$U_+$ and $U_-$. Therefore, by De Giorgi's structure theorem,
$$
    \cH^n\left(
      \bigl(\partial^*F\cap(U_+\cup U_-)\bigr)\setminus K
    \right)
    =
    \operatorname{Per}(F;O)
    =
    0.
$$

Thus the two boundary integrals above are supported in $K$.
Since $|Y_\pm(x)|\leq C_{\C}|x|\leq C_{\C}$ on $B_1$, the triangle inequality gives
\begin{equation*}
\begin{aligned}
    \cL^{n+1}(F\mathbin\triangle E)
    &\leq
      \int_{\partial^*F\cap U_+}
        |Y_+|\,|\nu_F-X|\,d\cH^n
      +\int_{\partial^*F\cap U_-}
        |Y_-|\,|\nu_F-X|\,d\cH^n\\
    &\leq C_{\C}
      \int_{\partial^*F\cap(U_+\cup U_-)}
        |\nu_F-X|\,d\cH^n.
\end{aligned}
\end{equation*}
Next, Cauchy--Schwarz and
$\cH^n(\partial^*F\cap B_1)=\Per(F;B_1)$ yield
\begin{align*}
    \cL^{n+1}(F\mathbin\triangle E)
    &\leq C_{\C}\Per(F;B_1)^{1/2}
      \left(
        \int_{\partial^*F\cap B_1}|\nu_F-X|^2\,d\cH^n
      \right)^{1/2}\\
    &\leq C_{\C}\Per(F;B_1)^{1/2}D^{1/2},
\end{align*}
where we use Lemma \ref{lem:calibration-identity} in the last inequality.

If $D\leq1$, then
$$
    \Per(F;B_1)=\Per(E;B_1)+D\leq\Per(E;B_1)+1 \le C_\C.
$$
Thus
\begin{equation}\label{eq:small-D}
    \cL^{n+1}(F\mathbin\triangle E)^2\leq C_{\C}D.
\end{equation}
If $D>1$, then we get the following trivial inequality:
$$
    \cL^{n+1}(F\mathbin\triangle E)^2
    \leq\cL^{n+1}(B_1)^2D.
$$
 This proves
\eqref{eq:main} for $R=1$.
\end{proof}

\subsection{Proof of Theorem \ref{thm:ordinary-Dirichlet-spectrum}}
\label{subsec:ordinary-jacobi-proof}
The arguments in \cite{de2014sharp,liu2019stability} use the geometry of
the Lawson cones.  Here we give a spectral proof that uses only the
stability of the regular hypercone.

\begin{proof}[Proof of Theorem~\ref{thm:ordinary-Dirichlet-spectrum}]
We first justify the equivalence of the two test classes appearing in
\eqref{eq:intro-ordinary-eigenvalue} and
\eqref{eq:test-class-equivalence}.  
Suppose first that $n>2$. We can use a cutoff argument, just as in the proof of \cite[Theorem 3]{de2014sharp}.

It remains to consider $n=2$.  Stability gives $b_1^2=\mu_1\geq0.$
Thus $\mu_1=0$ and $A_S\equiv0$.  Since the link is connected, it is a
great circle and $\C$ is a plane, so
\[
    Q_\C(\varphi)
    =
    \int_{\C\cap B_R}|\nabla_\C\varphi|^2\,d\cH^2.
\]
Define the Lipschitz cutoff
\[
    \chi_\varepsilon(r)
    =
    \begin{cases}
        0,
        &r\leq\varepsilon^2,\\[2mm]
        \displaystyle
        \frac{\log(r/\varepsilon^2)}{|\log\varepsilon|},
        &\varepsilon^2<r<\varepsilon,\\[3mm]
        1,
        &r\geq\varepsilon.
    \end{cases}
\]
After smoothing it near $r=\varepsilon^2$ and $r=\varepsilon$, we obtain
a $C^1$ cutoff, still denoted by $\chi_\varepsilon$, such that
\begin{equation}\label{eq:L^2 of grad cutoff}
    \int_\C|\nabla_\C\chi_\varepsilon|^2\,d\cH^2
    \leq\frac{C}{|\log\varepsilon|}
    \longrightarrow0.
\end{equation}
For
$\varphi_\varepsilon=\chi_\varepsilon\varphi$, dominated convergence and
\eqref{eq:L^2 of grad cutoff} give convergence in $L^2$ and in the
quadratic-form norm.  We have therefore proved
\begin{equation}\label{eq:test-class-equivalence}
    \lambda_\C^D(R)
    =
    \inf_{\substack{0\ne\varphi\in
        C_c^1((\C\setminus\{0\})\cap B_R)}}
      \frac{Q_\C(\varphi)}
      {\displaystyle\int_\C\varphi^2\,d\cH^n}.
\end{equation}

Combining \eqref{eq:test-class-equivalence} with
\eqref{eq:punctured-first-mode-reduction} and
\eqref{eq:radial-positive-lower-bound}, we obtain
\[
    \lambda_\C^D(R)
    =
    \Lambda_1(R)
    \geq
    \frac4{R^2}>0.
\]
It remains to identify the exact value of this one-dimensional infimum.
To construct a positive comparison solution for
\eqref{eq:first-radial-Rayleigh-quotient}, fix $\lambda>0$ and consider the Euler Lagrange equation of \eqref{eq:first-radial-Rayleigh-quotient}
\[
    -(r^{n-1}u')'+\mu_1r^{n-3}u
    =
    \lambda r^{n-1}u.
\]
Equivalently,
\begin{equation}\label{eq:CHS-radial-eigenvalue-equation}
    r^2u''+(n-1)ru'+(\lambda r^2-\mu_1)u=0.
\end{equation}

Next we transform \eqref{eq:CHS-radial-eigenvalue-equation} to a standard Bessel equation. Denote $\alpha = \frac{n-2}{2}.$ Consider the change of variables
\[
    z=\sqrt{\lambda }r,
    \qquad
    u(r)=r^{-\alpha}V(z).
\]
A direct computation gives
\[
    r^2u''+(n-1)ru'
    =
    r^{-\alpha}
    \left(z^2V''+zV'-\alpha^2V\right).
\]
Since $b_1^2=\alpha^2+\mu_1$, equation
\eqref{eq:CHS-radial-eigenvalue-equation} becomes
\[
    z^2V''+zV'+(z^2-b_1^2)V=0,
\]
which is the Bessel equation of order $b_1$.  The Bessel function of the
first kind (see, e.g., \cite[Chapter~10]{ahmad2015textbook}) has the expansion
\[
    J_{b_1}(z)
    =
    \frac{1}{\Gamma(b_1+1)}
    \left(\frac z2\right)^{b_1}
    \bigl(1+O(z^2)\bigr)
    \qquad\text{as }z\downarrow0.
\]
Consequently,
\[
    r^{-\alpha}J_{b_1}(\sqrt{\lambda} r)
    \sim c\,r^{-\alpha+b_1}
    =c\,r^{\gamma_+}
    \qquad\text{as }r\downarrow0.
\]
Thus, when $b_1>0$, the Bessel function $J_{b_1}$ selects exactly the
$\gamma_+$-branch of the homogeneous radial Jacobi equation used in
\cite{caffarelli1984minimal}; when $b_1=0$, it selects the
nonlogarithmic solution associated with the repeated indicial root.

Let $j_{b_1,1}$ be the first positive zero of $J_{b_1}$ and define
\begin{equation}\label{eq:first-radial-Bessel-solution}
    U_R(r)
    =
    r^{-\alpha}
    J_{b_1}\left(j_{b_1,1}\frac rR\right).
\end{equation}
Then
\[
    -(r^{n-1}U_R')'+\mu_1r^{n-3}U_R
    =
    \frac{j_{b_1,1}^2}{R^2}r^{n-1}U_R,
\]
and, because $j_{b_1,1}$ is the first positive zero,
\[
    U_R>0\quad\text{on }(0,R),
    \qquad
    U_R(R)=0.
\]
Hence $j_{b_1,1}^2/R^2$ is the natural candidate for
$\Lambda_1(R)$ and $ U_R$ the first (positive) solution.  The differential equation alone does not yet identify
the variational infimum because
$U_R\notin C_c^1((0,R))$.  We now prove the equality directly, using the
positive solution $U_R$ and cutoff approximations at the two endpoints.

First we denote
\[
    \lambda_0
    =
    \frac{j_{b_1,1}^2}{R^2}.
\]
The function $U_R$ defined in
\eqref{eq:first-radial-Bessel-solution} is positive on $(0,R)$ and
satisfies
\[
    -(r^{n-1}U_R')'
    +\mu_1r^{n-3}U_R
    =
    \lambda_0r^{n-1}U_R.
\]

For $0\ne u\in C_c^1((0,R))$, write
\[
    u=U_Rf,
    \qquad
    f=\frac{u}{U_R}.
\]

Expanding $u'$ and using the equation for $U_R$
gives
\begin{align*}
&r^{n-1}|u'|^2
+\mu_1r^{n-3}u^2
-\lambda_0r^{n-1}u^2\\
&=
r^{n-1}\left(U_R'f+U_Rf'\right)^2
+
\left(\mu_1r^{n-3}-\lambda_0r^{n-1}\right)
U_R^2f^2\\
&=
r^{n-1}U_R^2|f'|^2
+2r^{n-1}U_RU_R'ff'
+r^{n-1}|U_R'|^2f^2
+U_R\left(r^{n-1}U_R'\right)'f^2\\
&=
r^{n-1}U_R^2|f'|^2
+2r^{n-1}U_RU_R'ff'
+\left(r^{n-1}U_RU_R'\right)'f^2\\
&=
r^{n-1}U_R^2|f'|^2
+\left(r^{n-1}U_RU_R'f^2\right)'.
\end{align*}
After integration, the last term vanishes because $f$ is compactly
supported in $(0,R)$.  Therefore
\begin{equation}\label{eq:first-radial-positive-solution-identity}
\begin{aligned}
&\int_0^R
\left(
r^{n-1}|u'|^2+\mu_1r^{n-3}u^2
\right)\,dr
-\lambda_0\int_0^Rr^{n-1}u^2\,dr\\
&\qquad=
\int_0^R
r^{n-1}U_R^2
\left|
\left(\frac{u}{U_R}\right)'
\right|^2\,dr
\geq0.
\end{aligned}
\end{equation}
Taking the infimum over $u$ gives
\[
    \Lambda_1(R)\geq\lambda_0.
\]

For the reverse inequality, we approximate $U_R$ by functions in
$C_c^1((0,R))$.  Fix
\[
    0<\varepsilon<\min\{1,R/4\}.
\]
When $b_1>0$, choose a smooth cutoff
$\chi_\varepsilon^0:(0,R)\to[0,1]$ such that
\[
    \chi_\varepsilon^0=0
    \quad\text{on }(0,\varepsilon],
    \qquad
    \chi_\varepsilon^0=1
    \quad\text{on }[2\varepsilon,R),
    \qquad
    |(\chi_\varepsilon^0)'|
    \leq\frac{C}{\varepsilon}.
\]
When $b_1=0$, again we use instead a smoothed logarithmic cutoff which, before
smoothing, is given by
\[
    \chi_\varepsilon^0(r)
    =
    \begin{cases}
        0,
        &r\leq\varepsilon^2,\\[2mm]
        \displaystyle
        \frac{\log(r/\varepsilon^2)}{|\log\varepsilon|},
        &\varepsilon^2<r<\varepsilon,\\[3mm]
        1,
        &r\geq\varepsilon.
    \end{cases}
\]
After smoothing, we retain
\[
    0\leq\chi_\varepsilon^0\leq1,
    \qquad
    |(\chi_\varepsilon^0)'(r)|
    \leq\frac{C}{r|\log\varepsilon|}
\]
on the transition region.  Also choose a smooth cutoff
$\chi_\varepsilon^R:(0,R)\to[0,1]$ such that
\[
    \chi_\varepsilon^R=1
    \quad\text{on }(0,R-2\varepsilon],
    \qquad
    \chi_\varepsilon^R=0
    \quad\text{on }[R-\varepsilon,R),
    \qquad
    |(\chi_\varepsilon^R)'|
    \leq\frac{C}{\varepsilon}.
\]
Set
\[
    \eta_\varepsilon
    =
    \chi_\varepsilon^0\chi_\varepsilon^R,
    \qquad
    u_\varepsilon
    =
    \eta_\varepsilon U_R.
\]
After the indicated smoothing,
$u_\varepsilon\in C_c^1((0,R))$.

The expansion of $J_{b_1}$ at the origin gives
\[
    r^{n-1}U_R(r)^2
    \leq Cr^{1+2b_1}
    \qquad
    \text{for }r\text{ sufficiently small}.
\]
Moreover, $U_R(R)=0$ and $U_R$ is smooth near the regular endpoint
$r=R$, so
\[
    r^{n-1}U_R(r)^2
    \leq C(R-r)^2
    \qquad
    \text{for }r\text{ sufficiently close to }R.
\]

Because \(\eta_\varepsilon'
=
(\chi_\varepsilon^0)'\chi_\varepsilon^R
+
\chi_\varepsilon^0(\chi_\varepsilon^R)'\), \(|\eta_\varepsilon'|^2
\leq
2|(\chi_\varepsilon^0)'|^2
+
2|(\chi_\varepsilon^R)'|^2.\)
It follows that
\[
\int_0^R
r^{n-1}U_R^2|\eta_\varepsilon'|^2\,dr
\leq
\begin{cases}
C\bigl(\varepsilon^{2b_1}+\varepsilon\bigr),
&b_1>0,\\[2mm]
C\bigl(|\log\varepsilon|^{-1}+\varepsilon\bigr),
&b_1=0.
\end{cases}
\]
In either case, the right-hand side tends to zero.  On the other hand,
dominated convergence gives
\[
    \int_0^Rr^{n-1}u_\varepsilon^2\,dr
    \longrightarrow
    \int_0^Rr^{n-1}U_R^2\,dr
    >0.
\]
Applying \eqref{eq:first-radial-positive-solution-identity} to
$u_\varepsilon$ therefore yields
\[
\begin{aligned}
&\frac{\displaystyle
\int_0^R
\left(
r^{n-1}|u_\varepsilon'|^2
+\mu_1r^{n-3}u_\varepsilon^2
\right)\,dr}
{\displaystyle
\int_0^Rr^{n-1}u_\varepsilon^2\,dr}\\
&\qquad=
\lambda_0+
\frac{\displaystyle
\int_0^Rr^{n-1}U_R^2
|\eta_\varepsilon'|^2\,dr}
{\displaystyle
\int_0^Rr^{n-1}u_\varepsilon^2\,dr}
\longrightarrow\lambda_0.
\end{aligned}
\]
Thus
\[
    \Lambda_1(R)\leq\lambda_0.
\]
Combining the two inequalities and using
\eqref{eq:radial-positive-lower-bound}, we obtain
\begin{equation}\label{eq:first-radial-exact-value}
    \Lambda_1(R)
    =
    \frac{j_{b_1,1}^2}{R^2}
    \geq\frac4{R^2}>0.
\end{equation}

Hence we have proved that
\begin{equation}\label{eq:ordinary-Dirichlet-exact-value}
    \lambda_\C^D(R)
    =
    \frac{\lambda_\C^D(1)}{R^2}
    =
    \frac{j_{b_1,1}^2}{R^2}.
\end{equation}

Finally, it remains to identify the optimal constant in the
$|x|^{-2}$-weighted Jacobi inequality.  For
$\varphi\in C_c^1(\C\setminus\{0\})$, again by \eqref{eq:Dirichlet-mode-L2} and \eqref{eq:original-radial-Hardy-identity}
\[
\begin{aligned}
    Q_\C(\varphi)
    &=
    \sum_j\int_0^\infty
      \left(
        r\left|\left(r^\alpha u_j\right)'\right|^2
        +b_j^2r^{n-3}u_j^2
      \right)\,dr,\\
    \int_\C\frac{\varphi^2}{|x|^2}\,d\cH^n
    &=
    \sum_j\int_0^\infty r^{n-3}u_j^2\,dr.
\end{aligned}
\]
Since $b_j\geq b_1$, it follows that
\[
    Q_\C(\varphi)
    \geq
    b_1^2\int_\C\frac{\varphi^2}{|x|^2}\,d\cH^n.
\]
Finally, fix a nonzero $\eta\in C_c^1(\R)$ and set
\[
    \varphi_L(r,\omega)
    =
    r^{-\alpha}
    \eta\left(\frac{\log r}{L}\right)w_1(\omega).
\]
Then
\[
    \frac{Q_\C(\varphi_L)}
    {\displaystyle
      \int_\C\frac{\varphi_L^2}{|x|^2}\,d\cH^n}
    =
    b_1^2
    +
    \frac1{L^2}
    \frac{\displaystyle\int_\R|\eta'|^2}
         {\displaystyle\int_\R\eta^2}
    \longrightarrow b_1^2.
\]
Thus both constants asserted in the theorem are optimal, and
$b_1^2>0$ exactly when $\C$ is strictly stable.
\end{proof}

\begin{remark}
Consider the classical Simons cone
\[
    \C
    =
    M_{4,4}
    =
    \left\{
        (x,y)\in\R^4\times\R^4:
        |x|=|y|
    \right\}
    \subset\R^8.
\]
\cite[Theorem~3, equation~(1.22)]{de2014sharp}
provides the explicit lower bound
\[
    \lambda_\C^D(R)
    \geq
    \frac{\sqrt2}{16R^2}.
\]

On the other hand, the exact value can be computed from
Theorem~\ref{thm:ordinary-Dirichlet-spectrum}.  Here \(n=7\),
and the link is
\[
    S
    =
    \mathbb S^3\left(\frac1{\sqrt2}\right)
    \times
    \mathbb S^3\left(\frac1{\sqrt2}\right)
    \subset\mathbb S^7.
\]
Hence by \cite[p.~94]{simon1982isolated}, $\mu_1=-6,$ and
\[
    b_1
    =
    \sqrt{\frac{(n-2)^2}{4}+\mu_1}
    =
    \sqrt{\frac{25}{4}-6}
    =
    \frac12.
\]
The identity
\[
    J_{1/2}(z)
    =
    \sqrt{\frac{2}{\pi z}}\sin z
\]
shows that
\[
    j_{1/2,1}=\pi.
\]
Therefore by Theorem~\ref{thm:ordinary-Dirichlet-spectrum}, the
exact value is
\[
    \lambda_\C^D(R)
    =
    \frac{\pi^2}{R^2}.
\]
Thus the coefficient \(\sqrt2/16\approx0.0884\) obtained in
\cite{de2014sharp} is a valid but nonoptimal lower bound, whereas the
optimal coefficient is
\[
    \pi^2\approx9.8696.
\]
\end{remark}
\subsection{Proof of Theorem \ref{thm:weighted-characterization}}
\label{subsec:proof-weighted-characterization}

We shall repeatedly use the weight
$$
    q_\C(x)=\frac{\dist(x,\C)}{|x|^2}.
$$
Since $\C$ is a cone,
\begin{equation}\label{eq:weight-homogeneity}
    q_\C(tx)=t^{-1}q_\C(x)
    \qquad(t>0).
\end{equation}
Moreover, $q_\C(x)\leq|x|^{-1}$, so
$q_\C\in L^1_{\mathrm{loc}}(\R^{n+1})$. By \cite[Proposition~4.1]{de2014sharp} and
\cite[equation~(3)]{liu2019stability}, a quantitative subcalibration with
the appropriate signed bulk divergence gives the weighted quantitative
inequality.  We will follow this strategy for the sufficiency direction of
Theorem~\ref{thm:weighted-characterization}.  The new
task is to construct such a field without assuming the explicit symmetry of the
Lawson cones; this construction is carried out in
Lemma~\ref{lem:weighted-subcalibration}.  Suppose that $\C$ is strictly
stable and strictly minimizing.  We will construct a vector field $\mathcal X
    \in C^0(\R^{n+1}\setminus\{0\};\mathbb S^n),$
smooth in $U_+\cup U_-$, such that
\begin{align*}
    \mathcal X&=\nu_E
      &&\text{on }\C\setminus\{0\},\\
    \diver\mathcal X&\geq c_\C q_\C
      &&\text{in }U_+,\\
    \diver\mathcal X&\leq-c_\C q_\C
      &&\text{in }U_-,\\
    |\diver\mathcal X|&\leq C_\C q_\C
    &&\text{in }U_+\cup U_-.
\end{align*}

\begin{proof}[Proof of (a) $\Rightarrow$ (b)]
Let $\mathcal X$ be the field from
Lemma~\ref{lem:weighted-subcalibration}.  The quantitative-calibration
identity of \cite[Proposition~4.1]{de2014sharp}, in the notation used here,
is
\begin{align}
    D_R(F)
    ={}&\int_{F\mathbin\triangle E}
      |\diver\mathcal X|\,dx\notag\\
    &+\int_{\partial^*F\cap B_R}
      (1-\mathcal X\cdot\nu_F)\,d\cH^n.
      \label{eq:weighted-quantitative-calibration-identity}
\end{align}
The identity in \cite[Proposition~4.1]{de2014sharp} is stated for
$W^{1,1}_{\mathrm{loc}}$ fields. By the same mollification arguments as those used in Lemma~\ref{lem:calibration-identity}, it remains valid for the field in
Lemma~\ref{lem:weighted-subcalibration}: mollify $\mathcal X$ and use its
distributional divergence.  The mollifications converge uniformly on
compact subsets away from the origin, and the mollified divergences converge in
$L^1$ on compact sets because $\diver\mathcal X\in L^1_{\mathrm{loc}}$.
Thus the classical integration-by-parts identity passes to the limit and
gives \eqref{eq:weighted-quantitative-calibration-identity}.

The boundary term in
\eqref{eq:weighted-quantitative-calibration-identity} is nonnegative.
The signed estimates
\eqref{eq:weighted-subcalibration-plus} and 
\eqref{eq:weighted-subcalibration-minus} therefore give
\[
    D_R(F)
    \geq\int_{F\mathbin\triangle E}
       |\diver\mathcal X|\,dx
    \geq c_\C\mathcal W_R(F,E),
\]
as required.
\end{proof}

\begin{proof}[Proof of (b) $\Rightarrow$ (a)]
\textbf{Necessity of strict stability:} Fix $\varphi\in C_c^2(\C\setminus\{0\})$, and choose $R$ with
$\spt\varphi\Subset B_R$.  For small $t$, let $F_t$ agree with $E$
away from $\spt\varphi$ and have boundary
\[
    \{x+t\varphi(x)\nu_E(x):x\in\C\}
\]
near the support.  The perimeter expansion is
\begin{equation}\label{eq:weighted-normal-area-expansion}
    D_R(F_t)=\frac{t^2}{2}Q_\C(\varphi)+o(t^2).
\end{equation}
In a sufficiently small conical tubular neighborhood,
\[
    \dist(x+s\nu_E(x),\C)=|s|,
    \qquad
    |x+s\nu_E(x)|^2=|x|^2+s^2.
\]
The normal-coordinate Jacobian is $\det(I-sA_\C(x))$, and therefore
\begin{align}
    \mathcal W_R(F_t,E)
    &=
    \int_\C\int_0^{|t\varphi(x)|}
      \frac{s}{|x|^2+s^2}
      \det\left(
        I-\operatorname{sgn}(t\varphi(x))sA_\C(x)
      \right)ds\,d\cH^n(x)\notag\\
    &=\frac{t^2}{2}
      \int_\C\frac{\varphi^2}{|x|^2}\,d\cH^n+o(t^2).
      \label{eq:weighted-normal-W-expansion}
\end{align}
Apply \eqref{eq:weighted-inequality} and let
$t\to0$.  We obtain
\[
    Q_\C(\varphi)
    \geq c_\C\int_\C\frac{\varphi^2}{|x|^2}\,d\cH^n.
\]

\textbf{Necessity of strict minimality:}
We prove by contradiction. Suppose that $\C$ is not strictly minimizing.
By the strict-stability estimate proved above, we have $b_1^2\geq c_\C>0.$
By Proposition~\ref{prop:HS-root-criterion}, we assume without loss of generality that the fast leaf is
$\Sigma_+\subset U_+$ and has the fast exponent
\[
    \gamma
    =
    \gamma_-
    =
    -\frac{n-2}{2}-b_1,
    \qquad b_1>0.
\]

We suppress the sign and write the graph of $\Sigma$ by
\begin{equation}\label{eq:fast-leaf-expansion}
    G(r,\omega)
    =
    r\omega+h(r,\omega)\nu_E(\omega),
    \qquad
    h(r,\omega)
    =
    a r^\gamma w_1(\omega)
    +O_2(r^{\gamma-\varepsilon}),
\end{equation}
where $a>0$. We denote $\Omega_\Sigma$ the set such that
\[
    E\subset\Omega_\Sigma,
    \qquad
    \partial\Omega_\Sigma=\Sigma,
\]
with outer unit normal $\nu_\Sigma$.

Let $X$ be the foliation calibration from
Lemma~\ref{lem:calibration}. By Lemma~\ref{lem:calibration-identity}, after scaling we have
\begin{equation}\label{eq:fast-scaled-calibration-identity}
    D_\rho(F)
    =
    \frac12
    \int_{\partial^*F\cap B_\rho}
      |\nu_F-X|^2\,d\cH^n
\end{equation}
whenever $F\mathbin\triangle E\Subset B_\rho$.

Choose $\eta\in C^\infty([0,\infty);[0,1])$ such that
\[
    \eta=1\quad\text{on }[0,1],
    \qquad
    \eta=0\quad\text{on }[2,\infty),
\]
and define
\begin{equation}\label{eq:fast-leaf-cutoff}
    u_R(r,\omega)
    =
    \eta(r/R)h(r,\omega).
\end{equation}
Let $M_R$ be the hypersurface that coincides with $\Sigma$ for $r \leq R$ and is parametrized, for $r>R$, by the graph
\[
    G_R(r,\omega)
    =
    r\omega+u_R(r,\omega)\nu_E(\omega),
    \qquad
    r>R.
\]

For all sufficiently large $R$, the hypersurface $M_R$ bounds an open set $F_R$ satisfying
\begin{equation}\label{eq:fast-competitor-support}
    E\subset F_R,
    \qquad
    \partial F_R=M_R,
    \qquad
    F_R\mathbin\triangle E\Subset B_{4R}.
\end{equation}

Applying \eqref{eq:fast-scaled-calibration-identity} gives
\begin{equation}\label{eq:fast-cutoff-deficit-identity}
    D_{4R}(F_R)
    =
    \frac12
    \int_{M_R\cap B_{4R}}
      |\nu_{F_R}-X|^2\,d\cH^n.
\end{equation}
On the retained part of $\Sigma$ ($r<R$), we have
\[
    \nu_{F_R}=X=\nu_\Sigma,
\]
while on the retained part of $\C$ ($r>2R$), we have
\[
    \nu_{F_R}=X=\nu_E.
\]

Uniformly for $(r,\omega)\in[R,2R]\times S$,
\eqref{eq:fast-leaf-expansion} and
\eqref{eq:fast-leaf-cutoff} give
\begin{equation}\label{eq:fast-transition-data}
    \frac{|u_R(r,\omega)|}{r}
    +
    |\nabla_\C u_R(r,\omega)|
    \leq
    CR^{\gamma-1}.
\end{equation}
Then by \eqref{eq:fast-leaf-expansion} (or \cite[Lemma 4.6]{wang2020deformations}), we have
\begin{equation}\label{eq:fast-normal-cone}
    |\nu_{F_R}(G_R(r,\omega))-\nu_E(\omega)|
    \leq
    CR^{\gamma-1}.
\end{equation}

Next we evaluate $|X(G_R(r,\omega))-\nu_E(\omega)|$. If $u_R(r,\omega)=0$, then $G_R(r,\omega)\in\C$ and
$X(G_R(r,\omega))=\nu_E(\omega)$.  Otherwise, write the unique
foliation representation
\[
    G_R(r,\omega)
    =
    tG(\rho,\theta),
    \qquad t>0.
\]
% Since
% \[
%     \frac{\dist(G_R(r,\omega),\C)}
%          {|G_R(r,\omega)|}
%     \leq CR^{\gamma-1}\longrightarrow0
% \]
% and this ratio is invariant under dilations, the normalized point
% $G(\rho,\theta)$ lies on the graphical end when $R$ is sufficiently
% large.  
% Uniqueness of the conical normal coordinates gives
So we have
\[
    \theta=\omega,
    \qquad
    r=t\rho,
    \qquad
    u_R(r,\omega)=t h(\rho,\omega).
\]
Consequently,
\[
    \frac{|u_R(r,\omega)|}{r}
    =
    \frac{|h(\rho,\omega)|}{\rho}.
\]
Moreover,
\[
    X(G_R(r,\omega))
    =\nu_{t\Sigma}\bigl(tG(\rho,\omega)\bigr)=
    \nu_\Sigma(G(\rho,\omega)).
\]
By the $C^1$ fast asymptotics,
\[
\begin{aligned}
    |X(G_R(r,\omega))-\nu_E(\omega)|
    &\leq C\rho^{\gamma-1}\\
    &\leq C\frac{|h(\rho,\omega)|}{\rho}\\
    &=C\frac{|u_R(r,\omega)|}{r}\\
    &\leq CR^{\gamma-1},
\end{aligned}
\]

% The fast asymptotics then imply
% \[
%     |\nu_\Sigma(G(\rho,\omega))-\nu_E(\omega)|
%     \leq
%     C\rho^{\gamma-1}
%     \leq
%     C\frac{|h(\rho,\omega)|}{\rho}
%     =
%     C\frac{|u_R(r,\omega)|}{r},
% \]
% and
% \[
%     X(G_R(r,\omega))
%     =
%     \nu_\Sigma(G(\rho,\omega)).
% \]

Combining with \eqref{eq:fast-normal-cone}, we obtain
\begin{equation}\label{eq:fast-calibration-error}
    |\nu_{F_R}-X|
    \leq
    CR^{\gamma-1}
\end{equation}
uniformly for $r\in [R,2R]$. In addition, the area is at most $CR^n$, and hence
\eqref{eq:fast-cutoff-deficit-identity} gives
\begin{align}
    D_{4R}(F_R)
    &\leq
    CR^nR^{2\gamma-2}\notag\\
    &=
    CR^{n+2\gamma_--2}
    =
    CR^{-2b_1}.
    \label{eq:fast-deficit}
\end{align}

On the other hand, there is clearly a uniform weighted volume lower bound, i.e., there is a uniform $c_0>0$ such that
\begin{equation}\label{eq:fast-weight-lower}
    \mathcal W_{4R}(F_R,E)
    \geq
    c_0.
\end{equation}

Letting $R\to\infty$ leads to a contradiction, since $b_1>0$.
\end{proof}

\section{The vector fields and the identity properties}
\label{sec:field-construction}

The foliation idea in \cite[Section~1.4]{de2014sharp} is the starting point
for the calibration field in this section.  We first construct the field
$X$ and prove the exact formula for the perimeter deficit.  We then solve a
weighted divergence equation on each leaf $\Sigma_\pm$ and use it to
construct the fields $Y_\pm$. The resulting statements are given in Lemmas~\ref{lem:calibration}, \ref{lem:ambient-Y}, and~\ref{lem:calibration-identity}, and were used in the proofs in Section~\ref{sec:proof-idea}.

\subsection*{Construction of the vector fields}
We will construct vector fields $X,Y_\pm$ which satisfy the following properties.
\begin{lemma}\label{lem:calibration}
Define the vector field $X$ by 
\begin{align}
     X|_{S_\lambda}&=\nu_\lambda
        &&\text{for }\lambda\ne0,\label{eq:X-def}\\
    X&=\nu_E
        &&\text{on }\C\setminus\{0\},\label{eq:X-trace}
\end{align}
Then $X\in C^0(\R^{n+1}\setminus\{0\};\mathbb S^n)$
and
\begin{align}
    \diver X&=0
        &&\text{in }\mathcal D'(\R^{n+1}).\label{eq:X-div}
\end{align}
\end{lemma}

\begin{lemma}\label{lem:ambient-Y}
For each sign $\pm$, there is a smooth vector field $Y_\pm$ on $U_\pm$
such that
\begin{align}
    \diver Y_\pm&=1&&\text{in }U_\pm,\label{eq:Y-div}\\
    Y_\pm\cdot X&=0&&\text{in }U_\pm,\label{eq:Y-tangent}\\
    |Y_\pm(x)|&\leq C_{\C}|x|&&\text{for }x\in U_\pm.
      \label{eq:Y-bound}
\end{align}
\end{lemma}

\begin{lemma}[{cf. \cite[Proposition 4.1]{de2014sharp}}]\label{lem:calibration-identity}
Let $F$ be a set of locally finite perimeter with
$F\mathbin\triangle E\Subset B_1$.  Then
\begin{align}
    D_1(F)
    &=\int_{\partial^*F\cap B_1}
      (1-X\cdot\nu_F)\,d\cH^n\label{eq:calibration-identity}\\
    &=\frac12\int_{\partial^*F\cap B_1}
      |\nu_F-X|^2\,d\cH^n.\label{eq:angle-identity}
\end{align}
\end{lemma}

We first prove Lemma~\ref{lem:calibration}.

\begin{proof}[Proof of Lemma~\ref{lem:calibration}]
Clearly \eqref{eq:scaled-leaf-geometry} shows that $X$ is smooth on each side $U_\pm$. 
At $x\in S_\lambda$, $\lambda\ne0$, let
$e_1,\ldots,e_n$ be an orthonormal basis of $T_xS_\lambda$. We have
\begin{align*}
    \diver X(x)
    &=\sum_{i=1}^n\langle D_{e_i}X,e_i\rangle
      +\langle D_XX,X\rangle\\
    &=H_{S_\lambda}(x)=0.
\end{align*}
Thus
$\diver X=0$ classically in
$U_+\cup U_-$. And by \eqref{eq:leaf-normal-trace}, $X$ extends continuously across
$\C\setminus\{0\}$, thus \eqref{eq:X-trace} holds.

Next we show \eqref{eq:X-div}.
First consider $\varphi\in C_c^\infty(\R^{n+1}\setminus\{0\})$.  The outer normals of
$U_-=E$ and $U_+=\R^{n+1}\setminus\overline E$ on $\C$ are $\nu_E$ and
$-\nu_E$, respectively. Apply the divergence theorem first on smooth
inner exhaustions of the two sides and then pass to $\C$ using
the continuity of $X$.  This gives
\begin{align*}
\int_{\R^{n+1}}X\cdot\nabla\varphi\,dx
    &=\int_{\C}\varphi
      \left(
        X|_{\C}\cdot\nu_E-X|_{\C}\cdot\nu_E
      \right)d\cH^n=0.
\end{align*}
Thus $\diver X=0$ in
$\mathcal D'(\R^{n+1}\setminus\{0\})$.

Finally, let
$\varphi\in C_c^\infty(\R^{n+1})$ now be arbitrary, and let $\eta_\rho$ be a
smooth cutoff which is zero on $B_\rho$, one outside $B_{2\rho}$, and
satisfies $|\nabla\eta_\rho|\leq C/\rho$. By a standard capacity argument we get \eqref{eq:X-div}.
\end{proof}

We next prove the exact calibration identity. 

\begin{proof}[Proof of Lemma~\ref{lem:calibration-identity}]
Let $\rho_\varepsilon$ be a standard nonnegative mollifier and define
$X_\varepsilon=X*\rho_\varepsilon$.  By
Lemma~\ref{lem:calibration},
$\diver X_\varepsilon=0$.  Since $F\mathbin\triangle E\Subset B_1$, the vector-valued Radon measure
$D\chi_F-D\chi_E$ has compact support in $B_1$, and hence
$$
    \int_{\R^{n+1}}X_\varepsilon\cdot
    d(D\chi_F-D\chi_E)
    =-\int_{\R^{n+1}}(\chi_F-\chi_E)
      \diver X_\varepsilon\,dx
    =0.
$$
We now let $\varepsilon\downarrow0$.  Since $X$ is continuous on $\R^{n+1}\setminus\{0\}$, $X_\varepsilon(x)\to X(x)$ for every $x\ne0.$ In addition, $|X_\varepsilon|\leq1$, so by the dominated convergence theorem, 
$$
    \int_{\partial^*F\cap B_1}X\cdot\nu_F\,d\cH^n
    =\int_{\partial^*E\cap B_1}X\cdot\nu_E\,d\cH^n=\Per(E;B_1).
$$

Thus, 
$$ D_1(F)
    =\int_{\partial^*F\cap B_1}
      (1-X\cdot\nu_F)\,d\cH^n. $$
\end{proof}

\begin{lemma}\label{lem:leaf-Z}
For each sign $\pm$ there is a smooth vector field $Z_\pm$ tangent to
$\Sigma_\pm$ such that
\begin{equation}\label{eq:weighted-div}
    \diver_{\Sigma_\pm}(\psi_\pm Z_\pm)=\psi_\pm
    \qquad\text{on }\Sigma_\pm,
\end{equation}
and
\begin{equation}\label{eq:Z-growth}
    |Z_\pm(z)|\leq C_{\C}|z|
    \qquad\text{for every }z\in\Sigma_\pm.
\end{equation}
\end{lemma}

\begin{proof}
Fix one side and omit the subscript.
We enlarge $R_1>R_0$, if
necessary, so that the asymptotic estimates are valid for $r\geq R_1$.  For
$R\geq R_1$, use $\Gamma_R$ and $\eta_R$ from \eqref{eq:end-slice-definition}, and set
$$
    K_0=\Sigma\setminus G((R_1,\infty)\times S).
$$
When $R>R_1$, set
$$
    K_R=K_0\cup G([R_1,R]\times S).
$$
Thus $K_R$ is a compact domain with boundary $\Gamma_R$, and $\eta_R$ is
its outward unit conormal.
Choose the constant number
\begin{equation}\label{eq:flux-compatibility}
     a=\frac{1}{\cH^{n-1}(S)}
      \int_{K_0}\psi\,d\cH^n.
\end{equation}

On the graphical region, write
$\widehat\psi=\psi\circ G$ and use the radial tangent field $\partial_r$ defined
in \eqref{eq:partial_r}.  For $r\geq R_1$, define
\begin{equation}\label{eq:beta-def}
    \beta(r,\omega)
    =\frac{
        a+\displaystyle\int_{R_1}^r
          \widehat\psi(s,\omega)J(s,\omega)\,ds
    }{\widehat\psi(r,\omega)J(r,\omega)},
    \qquad
    Z_0=\beta\,\partial_r.
\end{equation}
The denominator is positive because $\widehat\psi>0$ by \eqref{eq:psi-def} and
$J>0$.  Moreover,
\eqref{eq:beta-def} is equivalent to
$$
    \widehat\psi(r,\omega)J(r,\omega)\beta(r,\omega)
    =a+\int_{R_1}^r
      \widehat\psi(s,\omega)J(s,\omega)\,ds,
$$
so the fundamental theorem of calculus gives
$$
    \partial_r(\widehat\psi J\beta)(r,\omega)
    =\widehat\psi(r,\omega)J(r,\omega).
$$
Taking $u=\widehat\psi\beta$ in \eqref{eq:end-radial-divergence}, we conclude
that
\begin{equation}\label{eq:end-div-computation}
    \bigl(\diver_\Sigma(\psi Z_0)\bigr)(G(r,\omega))
    =\frac{1}{J}\partial_r(J\widehat\psi\beta)(r,\omega)
    =\widehat\psi(r,\omega).
\end{equation}

Next, we verify the growth estimate for $\beta$.  Suppose first that
$\C$ is strictly stable. Let
$\gamma\in\{\gamma_-,\gamma_+\}$ denote the exponent in \eqref{eq:power-asymptotic}.  Combining \eqref{eq:power-asymptotic}, \eqref{eq:J-two-sided} and
\eqref{eq:HS-psi-comparison} gives
$$
    c r^{n-1+\gamma}
    \leq \widehat\psi(r,\omega)J(r,\omega)
    \leq C r^{n-1+\gamma}.
$$
By \eqref{eq:gamma-lower},
$\gamma\geq\gamma_-\geq-(n-2)$, and hence $n+\gamma\geq2$.  Thus
$$
    \int_{R_1}^r s^{n-1+\gamma}\,ds
    \leq C r^{n+\gamma}.
$$
It follows from \eqref{eq:beta-def} that
\begin{align*}
    |\beta(r,\omega)|
    &\leq
      \frac{
       a
        +C\displaystyle\int_{R_1}^r s^{n-1+\gamma}\,ds
      }{c r^{n-1+\gamma}}\\
    &\leq Cr.
\end{align*}

Suppose now that $\C$ is stable but not strictly stable.  Then by \eqref{eq:log-asymptotic}, \eqref{eq:J-two-sided}, and \eqref{eq:HS-psi-comparison}
$$
\begin{aligned}
    c r^{n/2}\bigl(1+|b|\log r\bigr)
    &\leq \widehat\psi(r,\omega)J(r,\omega)\\
    &\leq
    C r^{n/2}\bigl(1+|b|\log r\bigr).
\end{aligned}
$$
For $r\geq R_1>1$,
$$
    \int_{R_1}^r
      s^{n/2}\bigl(1+|b|\log s\bigr)\,ds
    \leq
    C r^{n/2+1}\bigl(1+|b|\log r\bigr).
$$
Hence \eqref{eq:beta-def} again yields
\begin{align*}
    |\beta(r,\omega)|
    &\leq
    \frac{
      a +C r^{n/2+1}(1+|b|\log r)
    }{
      c r^{n/2}(1+|b|\log r)
    }\\
    &\leq Cr.
\end{align*}
By the definition of $G$ in \eqref{eq:end-parametrization} and the decay of $h$, after increasing $R_1$ if necessary,
$|\partial_rG|\leq C$ and $|G(r,\omega)|\geq cr$.  Hence
\begin{equation}\label{eq:Z0-growth}
    |Z_0(G(r,\omega))|
    =|\beta(r,\omega)|\,|\partial_rG(r,\omega)|
    \leq Cr
    \leq C|G(r,\omega)|.
\end{equation}

Finally, we modify $Z_0$ on the compact part of $\Sigma$.  Choose $R_2>R_1$ and use a cutoff function to extend $Z_0$ to a smooth tangent vector field $\widetilde Z$ on $\Sigma$ such that $\widetilde Z=Z_0$ for $r\geq R_2$. Set
$$
    f=\psi-\diver_\Sigma(\psi\widetilde Z).
$$
Equation \eqref{eq:end-div-computation} shows that
$f\in C_c^\infty(\Sigma)$.  Take $R>R_2$ sufficiently large that
$\spt f\subset K_R$.  

Denote
\begin{equation}\label{eq:end-slice-definition}
     \Gamma_{\pm,R}=G_\pm(\{R\}\times S)
\end{equation}
and let $\eta_{\pm,R}$ be the unit conormal pointing toward increasing
$r$. By the definition of $J_\pm$ and the induced area element on
$\Gamma_{\pm,R}$,
\begin{equation}\label{eq:end-radial-flux}
    \int_{\Gamma_{\pm,R}}
    u\partial_r\cdot\eta_{\pm,R}\,d\cH^{n-1}
    =
    \int_S u(R,\omega)J_\pm(R,\omega)
      \,d\cH^{n-1}_S(\omega).
\end{equation}
Take \eqref{eq:end-radial-flux} with
$u=\widehat\psi\beta$,
\begin{equation}\label{eq:coordinate-flux}
    \int_{\Gamma_R}\psi Z_0\cdot\eta_R\,d\cH^{n-1}
    =\int_S \widehat\psi(R,\omega)J(R,\omega)\beta(R,\omega)
      \,d\cH^{n-1}.
\end{equation}
Since $\widetilde Z=Z_0$ on $\Gamma_R$, by the
divergence theorem,
\eqref{eq:beta-def}, and \eqref{eq:coordinate-flux}, we get
\begin{align*}
    \int_\Sigma f\,d\cH^n
    &=\int_{K_R}f\,d\cH^n\\
    &=\int_{K_R}\psi\,d\cH^n
      -\int_{\Gamma_R}\psi Z_0\cdot\eta_R\,d\cH^{n-1}\\
    &=\int_{K_0}\psi\,d\cH^n
      +\int_{R_1}^R\int_S
        \widehat\psi(s,\omega)J(s,\omega)
        \,d\cH^{n-1}\,ds\\
    &\quad-\int_S
      \left(a+\int_{R_1}^R        \widehat\psi(s,\omega)J(s,\omega)\,ds\right)
      d\cH^{n-1}\\
    &=0
\end{align*}
by \eqref{eq:flux-compatibility}.

As noted after \eqref{eq:foliation-map}, $\Sigma$ is connected and
oriented.  Let $dV_\Sigma=d\cH^n_\Sigma$ be its Riemannian volume
form.  Integration induces an isomorphism
\[
    H_c^n(\Sigma)\longrightarrow\R,
    \qquad
    [\omega]\longmapsto\int_\Sigma\omega;
\]
see \cite[Remark~5.7 and Corollary~5.8]{bott1982differential}.
The compactly supported top-degree form
\[
    \omega=f\,dV_\Sigma
\]
has zero integral. Hence there exists
$\alpha\in\Omega_c^{n-1}(\Sigma)$ such that
\[
    d\alpha=f\,dV_\Sigma.
\]

Hence there is a unique $W\in C_c^\infty(T\Sigma)$ satisfying
\[
    \iota_WdV_\Sigma=\alpha.
\]
Then by Cartan's formula,
\[
\begin{aligned}
    (\diver_\Sigma W)\,dV_\Sigma
    &=
    \mathcal L_WdV_\Sigma\\
    &=
    d(\iota_WdV_\Sigma)
      +\iota_W(d\,dV_\Sigma)\\
    &=
    d\alpha\\
    &=
    f\,dV_\Sigma.
\end{aligned}
\]
Therefore
\[
    \diver_\Sigma W=f.
\]
Since $\psi>0$ on $\Sigma$, the tangent field
\[
    V=\frac{W}{\psi}
\]
is smooth and compactly supported.  Define
\[
    Z=\widetilde Z+V.
\]
Then
\[
\begin{aligned}
    \diver_\Sigma(\psi Z)
    &=
    \diver_\Sigma(\psi\widetilde Z)
      +\diver_\Sigma W\\
    &=
    \diver_\Sigma(\psi\widetilde Z)+f\\
    &=\psi.
\end{aligned}
\]
Thus \eqref{eq:weighted-div} holds.

Outside a compact subset of $\Sigma$, we have $W=0$ and
$\widetilde Z=Z_0$.  Hence $Z=Z_0$ there, and
\eqref{eq:Z0-growth} gives
\[
    |Z(z)|\leq C|z|.
\]
On the remaining compact subset, $Z$ is bounded and $|z|$ has a
positive lower bound because $0\notin\Sigma$.  After increasing the
constant, the same estimate holds on all of $\Sigma$.  This proves
\eqref{eq:Z-growth}.
\end{proof}
Finally, we define $Y_\pm$ by scaling $Z_\pm$ along the foliation leaves.

\begin{proof}[Proof of Lemma~\ref{lem:ambient-Y}]
For each $(t,z)\in (0,\infty)\times \Sigma_\pm$, we define
\begin{equation}\label{eq:Y-def}
    Y_\pm(tz)=tZ_\pm(z).
\end{equation}
This is smooth because $\Phi_\pm$ is a diffeomorphism. By \eqref{eq:scaled-field-divergence} with $Z=Z_\pm$ and \eqref{eq:weighted-div}, we have
$$
    (\diver Y_\pm)(tz)
    =\frac{1}{\psi_\pm(z)}
      \diver_{\Sigma_\pm}(\psi_\pm Z_\pm)(z)
    =1,
$$
which proves \eqref{eq:Y-div}.  

Finally, \eqref{eq:Z-growth} gives
$$
    |Y_\pm(tz)|
    =t|Z_\pm(z)|
    \leq Ct|z|
    =C|tz|,
$$
which proves \eqref{eq:Y-bound}.
\end{proof}

\subsection*{The volume flux identity}
\label{subsec:bv-proofs}

In this subsection, we will prove the volume-flux formula used in
the proof of Theorem \ref{thm:main}.

\begin{lemma}\label{lem:volume-flux}
Let $U$ be either $U_+$ or $U_-$, and let $Y$ be the corresponding field
from Lemma~\ref{lem:ambient-Y}.  If $A\subset U$ has finite
perimeter and $\spt\chi_A\Subset B_1$, then
\begin{equation}\label{eq:volume-flux}
    \cL^{n+1}(A)
    =\int_{\partial^*A\cap U}Y\cdot\nu_A\,d\cH^n.
\end{equation}
\end{lemma}

\begin{proof}[Proof of Lemma~\ref{lem:volume-flux}]
For simplicity, we write $U=U_\pm$, $\Sigma=\Sigma_\pm$, and
$\lambda=\lambda_\pm$.  By \eqref{eq:lambda-def}, $\lambda$ is smooth
without critical points and $\{\lambda=s\}=s\Sigma$ for every $s>0$.
Equations
\eqref{eq:lambda-gradient} and \eqref{eq:X-def} also show that
$\nabla\lambda$ is parallel to $X$.

For $s>0$, set
$$
    H_s=\{\lambda>s\},
    \qquad
    A_s=A\cap H_s,
$$
so $\partial H_s=s\Sigma$ and $
    \spt\chi_{A_s}
    \subset\overline B_1\cap\{\lambda\geq s\}\Subset U.$

Next, we choose a regular level of the foliation parameter.
Since
$\cH^n(\partial^*A\cap B_1)<\infty$, all but at most countably many
$s>0$ satisfy
\begin{equation}\label{eq:good-slice}
    \cH^n(\partial^*A\cap\{\lambda=s\})=0.
\end{equation}
Fix such an $s$.  By
\cite[Theorem~16.3]{maggi2012sets},
$A_s$ has finite perimeter and, up to an $\cH^n$-null set,
\begin{equation}\label{eq:As-boundary-decomposition}
    \partial^*A_s
    =\bigl(\partial^*A\cap H_s\bigr)
      \sqcup
      \bigl(A^{(1)}\cap\partial H_s\bigr).
\end{equation}
Hence on the
first part,
$\nu_{A_s}=\nu_A$, and on the second part,
\begin{equation}\label{eq:slice-normal}
    \nu_{A_s}=\nu_{H_s}
    =-\frac{\nabla\lambda}{|\nabla\lambda|},
\end{equation}
which is parallel to $X$.

Because $\diver Y=1$, we have
\begin{align}
    \cL^{n+1}(A_s)
    &=\int_{A_s}\diver(Y)\,dx \nonumber\\
    &=\int_{\partial^*A_s}
      Y\cdot\nu_{A_s}\,d\cH^n \nonumber\\
    &=\int_{\partial^*A\cap\{\lambda>s\}}
      Y\cdot\nu_{A_s}\,d\cH^n, \label{eq:truncated-volume-flux}
\end{align}
 where the last equality uses \eqref{eq:slice-normal}. 

Finally, choose regular values $s_j\downarrow0$. Monotone convergence gives
$$
    \cL^{n+1}(A\cap\{\lambda>s_j\})
    \longrightarrow \cL^{n+1}(A).
$$
On the boundary side, $\mathbf1_{\{\lambda>s_j\}}Y\cdot\nu_A
$ converges pointwise on $\partial^*A\cap U$ to $Y\cdot\nu_A$. Moreover,
$|Y(x)|\leq C_{\C}|x|\leq C_{\C}$ on $B_1$, while
$\cH^n(\partial^*A\cap B_1)<\infty$.  By the dominated convergence theorem,
\eqref{eq:truncated-volume-flux} therefore proves
\eqref{eq:volume-flux}.
\end{proof}

This completes the proof of Theorem~\ref{thm:main}.

\section{The weighted quantitative characterization}
\label{sec:weighted-characterization}

In this section, we use the foliations to construct the quantitative subcalibration needed for the sufficiency direction of Theorem~\ref{thm:weighted-characterization}.

\subsection*{A weighted field on \texorpdfstring{$\Sigma_\pm$}{Sigma plus/minus}}
Here and below, if
$Z\in C^1(\Sigma_\pm;T\Sigma_\pm)$, we regard $Z$ as an
$\R^{n+1}$-valued map on $\Sigma_\pm$ and write
\[
    dZ_z:T_z\Sigma_\pm\longrightarrow\R^{n+1},
    \qquad
    dZ_z[\xi]
    :=
    D_\xi^{\R^{n+1}}Z.
\]

\begin{lemma}
\label{lem:weighted-leaf-field}
Assume that $\C$ is strictly stable and strictly minimizing.  For each
sign there is a smooth tangent field
$Z_\pm^{\mathrm w}$ on $\Sigma_\pm$ such that
\begin{equation}\label{eq:weighted-leaf-divergence}
    \diver_{\Sigma_\pm}(\psi_\pm Z_\pm^{\mathrm w})
    =\frac{\psi_\pm^2}{|z|^2}.
\end{equation}
Moreover,
\begin{equation}\label{eq:weighted-leaf-estimates}
    |Z_\pm^{\mathrm w}(z)|
    \leq C_\C\frac{\psi_\pm(z)}{|z|},
    \qquad
    |dZ_\pm^{\mathrm w}(z)|
    \leq C_\C\frac{\psi_\pm(z)}{|z|^2}.
\end{equation}
\end{lemma}

\begin{proof}
We omit the sign and denote
$\gamma=\gamma_+=-(n-2)/2+b_1$.  
Define the smooth positive function
\begin{equation}\label{eq:leaf-smooth-weight}
    \rho(z)=\frac{\psi(z)}{|z|^2}.
\end{equation}
Hence by \eqref{eq:HS-psi-comparison},
\begin{equation}\label{eq:rho-distance-comparison}
    c_\C\frac{\dist(z,\C)}{|z|^2}
    \leq\rho(z)
    \leq C_\C\frac{\dist(z,\C)}{|z|^2}
    \qquad\text{on }\Sigma.
\end{equation}
We use $\rho$ rather than the distance function itself, so
that the right-hand side of \eqref{eq:weighted-leaf-divergence} is smooth.

Write
$\widehat\psi=\psi\circ G$ and use $R_1$ from
Proposition~\ref{prop:HS-root-criterion}.  
Denote
\[
    K_0=\Sigma\setminus G((R_1,\infty)\times S),
\]
and
\begin{equation}\label{eq:weighted-flux-constant}
    a_0
    =\frac{1}{\cH^{n-1}(S)}
      \int_{K_0}\psi\rho\,d\cH^n>0.
\end{equation}
Then we define
\[
    H(r,\omega)
    =
    a_0+
    \int_{R_1}^r
      \frac{\widehat\psi(s,\omega)^2J(s,\omega)}
           {|G(s,\omega)|^2}\,ds
\]
and set
\begin{equation}
    \beta(r,\omega)
    =
    \frac{H(r,\omega)}
         {\widehat\psi(r,\omega)J(r,\omega)},
    \qquad
    Z_0=\beta\,\partial_rG.
    \label{eq:weighted-beta-def}
\end{equation}
Thus the radial divergence formula \eqref{eq:end-radial-divergence} gives the
exact identity
\begin{equation}\label{eq:weighted-end-equation}
\begin{aligned}
    \diver_\Sigma(\psi Z_0)
    &=\frac1J\partial_r(J\widehat\psi\beta)\\
    &=\frac1J\partial_rH\\
    &=\frac{\widehat\psi^2}{|G|^2}
      =\psi\rho.
\end{aligned}
\end{equation}

Note that by \eqref{eq:HS-psi-comparison} and Proposition~\ref{prop:HS-root-criterion}, for $(r,\omega)\in[R_1,\infty)\times S$,
\begin{equation}\label{eq:weighted-H-estimates}
    c r^{2b_1}\leq H(r,\omega)\leq C r^{2b_1},
    \qquad
    |\nabla_SH(r,\omega)|\leq Cr^{2b_1}, \qquad \partial_rH
    =
    \frac{\widehat\psi^2J}{|G|^2}.
\end{equation}
This is the precise point at which strict stability is used. Then a direct computation shows that
\begin{equation}\label{eq:weighted-beta-estimates}
    |\beta|+r|\partial_r\beta|+|\nabla_S\beta|
    \leq Cr^{\gamma-1}.
\end{equation}
The weighted $C^2$ estimate \eqref{eq:slow-weighted-C2} thus gives
$|\partial_rG|\leq C$ and
$|d(\partial_rG)|\leq C/r$. Hence by the product rule,
\begin{equation}\label{eq:weighted-Z-end-estimates}
    |Z_0|+r|dZ_0|
    \leq Cr^{\gamma-1}.
\end{equation}

It remains to extend the field across the compact region. This follows by exactly the same argument as in the proof of Lemma~\ref{lem:leaf-Z}, so we omit the details.
\end{proof}
\subsection*{The quantitative subcalibration}
Let $Z\in C^1(\Sigma_\pm;T\Sigma_\pm)$ and define the degree-zero field
$W$ by $W(tz)=Z(z)$.  In the coordinates of $\Phi_\pm$, the field $W$
has components $(0,t^{-1}Z)$.  Since the coordinate volume density is
$t^n\psi_\pm$ by \eqref{eq:Phi-Jacobian},
\begin{equation}\label{eq:zero-homogeneous-field-divergence}
\begin{aligned}
    (\diver W)(tz)
    &=\frac{1}{t^n\psi_\pm(z)}
      \diver_{\Sigma_\pm}
      \bigl(t^{n-1}\psi_\pm Z\bigr)(z)\\
    &=\frac{1}{t\psi_\pm(z)}
      \diver_{\Sigma_\pm}(\psi_\pm Z)(z).
\end{aligned}
\end{equation}

Next we compute the Euclidean derivative of \(W\).  Since \(W\) is
defined through the foliation coordinates \(\Phi_\pm\), we first
express an arbitrary ambient direction in these coordinates.
Fix \(x=tz\) and \(v\in\mathbb R^{n+1}\).  Let \(v^\top\) denote the
orthogonal projection of \(v\) onto \(T_z\Sigma_\pm\), and put
\[
    c(z)=z\cdot\nu_\pm(z),
    \qquad
    z^\top=z-c(z)\nu_\pm(z).
\]
Since \(c(z)\neq0\), the inverse of
\eqref{eq:Phi-differential} is
\[
    (D\Phi_\pm)_{(t,z)}^{-1}v
    =
    \left(
      \frac{v\cdot\nu_\pm(z)}{c(z)},
      \frac1t
      \left(
        v^\top-
        \frac{v\cdot\nu_\pm(z)}{c(z)}z^\top
      \right)
    \right).
\]

Since
\[
    W\circ\Phi_\pm(t,z)=Z(z),
\]
the field \(W\circ\Phi_\pm\) is independent of \(t\).  The chain rule
therefore gives
\begin{equation}\label{eq:zero-homogeneous-field-derivative}
    DW_{tz}[v]
    =
    \frac1t
    dZ_z
    \left(
      v^\top-
      \frac{v\cdot\nu_\pm(z)}{c(z)}z^\top
    \right).
\end{equation}

\begin{lemma}
\label{lem:degree-zero-weighted-field}
Under the assumptions of Lemma~\ref{lem:weighted-leaf-field}, define
\begin{equation}\label{eq:degree-zero-lift}
    W_\pm(tz)=Z_\pm^{\mathrm w}(z),
    \qquad t>0,\quad z\in\Sigma_\pm.
\end{equation}
Then
\begin{align}
    W_\pm\cdot X&=0,
      \label{eq:degree-zero-tangent}\\
    c_\C q_\C(x)
    \leq\diver W_\pm(x)
    &\leq C_\C q_\C(x),
      \label{eq:degree-zero-divergence}\\
    |W_\pm(x)|
    &\leq C_\C\frac{\dist(x,\C)}{|x|},
      \label{eq:degree-zero-size}\\
    |DW_\pm(x)|
    &\leq\frac{C_\C}{|x|}.
      \label{eq:degree-zero-derivative}
\end{align}
Moreover, $W_\pm(x)\to0$ as $x\in U_\pm$ approaches
$\C\setminus\{0\}$.
\end{lemma}

\begin{proof}
Put $\gamma=\gamma_+=-(n-2)/2+b_1$.
 By the degree-zero equation
\eqref{eq:zero-homogeneous-field-divergence} and
\eqref{eq:weighted-leaf-divergence}, 
\begin{equation}\label{eq:degree-zero-divergence-exact}
    (\diver W_\pm)(tz)
    =\frac1t\frac{\psi_\pm(z)}{|z|^2}.
\end{equation}
Since $\dist(tz,\C)=t\dist(z,\C),$
the comparison \eqref{eq:HS-psi-comparison} proves
\eqref{eq:degree-zero-divergence}.  The same comparison and
\eqref{eq:weighted-leaf-estimates} give
\[
    |W_\pm(tz)|
    \leq C\frac{\dist(z,\C)}{|z|}
    =C\frac{\dist(tz,\C)}{|tz|},
\]
which proves \eqref{eq:degree-zero-size} and shows that $W_\pm$
extends continuously to $\C\setminus\{0\}$ by setting $W_\pm=0$
there.

We next include the derivative estimate. For
$v\in\R^{n+1}$, we have
\[
    DW_{\pm,tz}[v]
    =\frac1tdZ_\pm^{\mathrm w}
      \left(
        v^\top-\frac{v\cdot\nu_\pm}{c(z)}z^\top
      \right).
\]
On the end, $|z^\top|\leq Cr$, $|c(z)|=\psi_\pm(z)\geq cr^\gamma$,
and
$|dZ_\pm^{\mathrm w}|
\leq Cr^{\gamma-2}$.  Hence
\[
    |DW_{\pm,tz}[v]|
    \leq\frac Ct r^{\gamma-2}(1+r^{1-\gamma})|v|
    \leq\frac C{tr}|v|
    \leq\frac C{|tz|}|v|.
\]
On the compact part of $\Sigma_\pm$, the factor
$1+|z^\top|/\psi_\pm$ is bounded and
$t^{-1}=|z|/|tz|$.  Hence combining together we prove
\eqref{eq:degree-zero-derivative} globally.
\end{proof}

\begin{lemma}
\label{lem:weighted-subcalibration}
Assume that $\C$ is strictly stable and strictly minimizing.  There is
\[
    \mathcal X
    \in C^0(\R^{n+1}\setminus\{0\};\mathbb S^n),
\]
smooth in $U_+\cup U_-$, such that
\begin{align}
    \mathcal X&=\nu_E
      &&\text{on }\C\setminus\{0\},
      \label{eq:weighted-subcalibration-trace}\\
    \diver\mathcal X&\geq c_\C q_\C
      &&\text{in }U_+,
      \label{eq:weighted-subcalibration-plus}\\
    \diver\mathcal X&\leq-c_\C q_\C
      &&\text{in }U_-,
      \label{eq:weighted-subcalibration-minus}\\
       |\diver \mathcal X|
    &\leq C_\C q_\C
      &&\text{in }U_+\cup U_-.
      \label{eq:weighted-subcalibration-absolute}
\end{align}
Moreover,
\begin{equation}
\label{eq:weighted-subcalibration-distributional-divergence}
    \diver \mathcal X
    =
    g
    \qquad\text{in }\mathcal D'(\R^{n+1}),
\end{equation}
where
\[
    g(x)
    =
    \begin{cases}
    \diver \mathcal X,
       &x\in U_+\cup U_-,\\[2ex]
    0,
       &x\in\C,
    \end{cases}
\]
and $g
    \in
    C^0(\R^{n+1}\setminus\{0\})
    \cap
    L^1_{\mathrm{loc}}(\R^{n+1}).$
\end{lemma}

\begin{proof}
Fix $\delta>0$, to be chosen sufficiently small, and define
\begin{equation}\label{eq:weighted-subcalibration-def}
    \mathcal X(x)
    =
    \begin{cases}
    \displaystyle
    \frac{X(x)+\delta W_+(x)}
         {\sqrt{1+\delta^2|W_+(x)|^2}},
       &x\in U_+,\\[2ex]
    \displaystyle
    \frac{X(x)-\delta W_-(x)}
         {\sqrt{1+\delta^2|W_-(x)|^2}},
       &x\in U_-,\\[2ex]
    \nu_E(x),
       &x\in\C\setminus\{0\}.
    \end{cases}
\end{equation}

Because $W_\pm\cdot X=0$, the field $\mathcal X$ has unit length.
Moreover, for every $x\in\C\setminus\{0\}$, the continuous extensions
of $X$ and $W_\pm$ satisfy
\[
    X(x)=\nu_E(x),
    \qquad
    W_\pm(x)=0.
\]
Consequently,
\[
    \lim_{\substack{y\to x\\y\in U_+}}\mathcal X(y)
    =
    \lim_{\substack{y\to x\\y\in U_-}}\mathcal X(y)
    =
    \nu_E(x).
\]
Thus $\mathcal X$ extends continuously across
$\C\setminus\{0\}$, and
\eqref{eq:weighted-subcalibration-trace} holds.

Let
\[
    A_\pm
    =
    (1+\delta^2|W_\pm|^2)^{-1/2}.
\]
Since $\diver X=0$ on $U_+\cup U_-$, we have
\begin{align}
    \diver(\mathcal X)
    &=
    \pm\delta A_\pm\diver W_\pm\notag\\
    &\quad-
    \frac{\delta^2}{2}A_\pm^3
    (X\pm\delta W_\pm)\cdot\nabla|W_\pm|^2.
    \label{eq:normalized-divergence}
\end{align}
By \eqref{eq:degree-zero-size} and
\eqref{eq:degree-zero-derivative},
\begin{equation}\label{eq:normalization-error}
    |\nabla|W_\pm|^2|
    \leq
    2|W_\pm||DW_\pm|
    \leq
    C_\C q_\C.
\end{equation}
Also $|W_\pm|\leq C_\C$.  The first term in
\eqref{eq:normalized-divergence} has size at least
$c\delta q_\C$ with the indicated sign, while the error term is bounded
in absolute value by $C\delta^2q_\C$.  Choosing $\delta$ sufficiently
small we get
\eqref{eq:weighted-subcalibration-plus} and
\eqref{eq:weighted-subcalibration-minus}.  And combining
with the upper bound in \eqref{eq:degree-zero-divergence}, we prove
\eqref{eq:weighted-subcalibration-absolute}.

It remains to prove
\eqref{eq:weighted-subcalibration-distributional-divergence}.
First let
\[
    \varphi
    \in
    C_c^\infty(\R^{n+1}\setminus\{0\}).
\]
As in Lemma~\ref{lem:calibration}, we first apply the divergence theorem
to smooth exhaustions of \(U_+\) and \(U_-\).  For every
\(\varphi\in C_c^\infty(\R^{n+1}\setminus\{0\})\), this gives
\begin{align}
    -\int_{\R^{n+1}}
      \mathcal X\cdot\nabla\varphi\,dx
    &=
      \int_{\R^{n+1}}g\varphi\,dx +
      \int_\C\varphi
      (\mathcal X|_\C-\mathcal X|_\C)\cdot\nu_E
      \,d\cH^n.
      \label{eq:punctured-distributional-divergence}
\end{align}
 Hence
\[
    -\int_{\R^{n+1}}
      \mathcal X\cdot\nabla\varphi\,dx
    =
    \int_{\R^{n+1}}g\varphi\,dx.
\]

Now let
$\varphi\in C_c^\infty(\R^{n+1})$ be arbitrary.  Choose a smooth cutoff
$\eta_\rho$ which vanishes on $B_\rho$, equals one outside
$B_{2\rho}$, and satisfies $
    |\nabla\eta_\rho|
    \leq
    \frac C\rho.$
Applying \eqref{eq:punctured-distributional-divergence} to
$\eta_\rho\varphi$, then a standard capacity argument proves
\eqref{eq:weighted-subcalibration-distributional-divergence}.
\end{proof}

Together with the argument in
Subsection~\ref{subsec:proof-weighted-characterization}, this completes
the proof of Theorem~\ref{thm:weighted-characterization}.

{\small
\bibliographystyle{alpha}
\bibliography{main}
}

\bigskip

\begingroup
\small
\normalfont
\noindent
Department of Mathematics, University of Rochester\\
Rochester, New York\\
Email address:
\href{mailto:gniu3@ur.rochester.edu}
     {\texttt{gniu3@ur.rochester.edu}}
\par
\endgroup
\end{document}